\documentclass[10pt]{amsart}

\newtheorem{thm}{Theorem}[subsection]
\newtheorem{lem}[thm]{Lemma}
\newtheorem{cor}[thm]{Corollary}
\newtheorem{prop}[thm]{Proposition}

\theoremstyle{definition}
\newtheorem{defn}[thm]{Definition}
\newtheorem{eg}[thm]{Example}

\theoremstyle{remark}
\newtheorem{rem}[thm]{\bf Remark}

\numberwithin{equation}{section}

\usepackage{amscd}
\usepackage{amssymb}
\usepackage{graphicx}
\usepackage{color}
\usepackage{epstopdf}
\usepackage{hyperref}

\def\vs#1{\vskip .#1 cm} 
\def\noi{\noindent}

\def\larrow{\leftarrow}
\def\rarrow{\rightarrow}

\def\xrarrow{\xrightarrow} 

\def\smallcoprod{\,{\textstyle{\coprod}}\,}

\def\noteq{\neq}

\def\<{\left<}
\def\>{\right>}

\DeclareMathOperator{\mmod}{mod}
\DeclareMathOperator{\ind}{ind}
\DeclareMathOperator{\im}{im}
\DeclareMathOperator{\coker}{coker}
\DeclareMathOperator{\Hom}{Hom}%
\DeclareMathOperator{\Ext}{Ext}%

\DeclareMathOperator{\Aut}{Aut}

\DeclareMathOperator{\End}{End}

\newcommand{\field}[1]{\mathbb{#1}}
\newcommand{\ZZ}{\ensuremath{{\field{Z}}}}

\newcommand{\RR}{\ensuremath{{\field{R}}}}
\newcommand{\QQ}{\ensuremath{{\field{Q}}}}
\newcommand{\NN}{\ensuremath{{\field{N}}}}

\newcommand{\length}[1]{\ensuremath {\| #1 \|}}

\newcommand{\commentout}[1]{}

\def\ll{\lambda}
\def\LL{\Lambda}

\newcommand{\cc}{{C}}

\newcommand{\cT}{\ensuremath{{\mathcal{T}}}}

\def\a{\alpha}
\def\b{\beta}
\def\g{\gamma}

\def\f{\phi}
\def\k{\kappa}

\def\s{\sigma}

\def\t{\tau}

\def\z{\zeta}

\def\then{\Rightarrow}
\def\into{\hookrightarrow}
\def\onto{\twoheadrightarrow}

\def\ul{\underline}

\def\k{\Bbbk}
\def\maptoM{\pi}
\def\elemofU{\psi}

\def\st{\,|\,}

\def\Eti{(E^t)^{-1}}

\begin{document}


\title{Cluster complexes via semi-invariants}

\author{Kiyoshi Igusa}
\address{Brandeis University}
\email{igusa@brandeis.edu}

\author{Kent Orr}
\address{Indiana University}
\email{korr@indiana.edu}
\thanks{Supported by NSF grant DMS-0406563.}

\author{Gordana Todorov}
\address{Northeastern University}
\email{todorov@neu.edu}
\thanks{Supported by NSA}

\author{Jerzy Weyman}
\address{Northeastern University}
\email{j.weyman@neu.edu}
\thanks{Supported by NSF grant DMS-0600229}

\date{\today}

\begin{abstract}
We define and study virtual representation spaces for vectors having both positive and negative dimensions at the vertices of a quiver without oriented cycles. We consider the natural semi-invariants on these spaces which we call {\em virtual semi-invariants} and prove that they satisfy the three basic theorems: the First Fundamental Theorem, the Saturation Theorem and the Canonical Decomposition Theorem.

In the special case of Dynkin quivers with $n$ vertices this gives the fundamental interrelationship between supports of the semi-invariants and the Tilting Triangulation of the $(n-1)$-sphere.
\end{abstract}

\maketitle

\def\repspace{presentation space}
\setcounter{section}{-1}
\tableofcontents

\section{Introduction} 
This paper initiates a project to apply quiver representations and their semi-invariants to expose compatible combinatorial underpinnings for the tilting objects of cluster categories (and hence, clusters for cluster algebras), and for the homology of nilpotent groups.  Here we focus on semi-invariants and tilting objects in cluster categories, by extending the classical semi-invariant results of Kac, Schofield, Derksen and Weyman, and interpreting the fundamental results about cluster categories from \cite{[BMRRT]} to this setting.  

Modeling from K-theory, for an arbitrary quiver without oriented cycles, we consider semi-invariants in the derived category by extending the definition of representation spaces to {\em virtual dimension vectors} of virtual modules over the path algebra of the quiver.  Such virtual dimension vectors have both positive and negative coordinates.
 Specifically, instead of working with representation spaces of the quiver acted upon by products of general linear groups, we work with {\em presentation spaces}, the spaces $\Hom_Q (P_1,P_0)$ for projective modules $P_0, P_1$. The natural action of the group  $\Aut(P_0)\times (\Aut(P_1))^{op}$  replaces the action of the product of general linear groups and we study the semi-invariants for these actions. 
 We construct the {\em virtual representation space} for virtual dimension vectors $\a \in \ZZ^n$ as a direct limit 
 $$R^{vir}(\a) = \lim_{P} \Hom_Q(P_1\smallcoprod P, P_0 \smallcoprod P),$$
where $\a = \ul\dim P_0 - \ul\dim P_1$ and the direct limit is over all projectives $P$. The natural semi-invariants are obtained as inverse limits of semi-invariants on the presentation spaces. We call them {\em virtual semi-invariants}. 

We prove the three basic theorems in the virtual setting. The Virtual First Fundamental Theorem~\ref{thm: vFFT} relates virtual semi-invariants to quiver representations, i.e. all virtual semi-invariants are linear combinations  of determinantal semi-invariants. The Virtual Saturation Theorem~\ref{virsatthm} describes the supports of semi-invariants, i.e. describes when the determinantal semi-invariants are non-zero. The Virtual Generic Decomposition Theorem~\ref{thm:generic decomposition} determines the dimension vectors of the indecomposable components of all generic representations of all dimension vectors.

Using the above results about virtual semi-invariants, to each quiver with $n$ vertices, we associate a simplicial complex $\cT(Q)$ together with a mapping of its geometric realization to the $(n-1)$-dimensional sphere
\[
	\ll:|\cT(Q)|\to S^{n-1}.
\]
The simplices of $\cT(Q)$ are virtual partial tilting sets of Schur roots and shifted projective roots.
We call this complex the {\em Complex of Virtual Tilting Sets}. In general it has infinitely many simplices. The continuous mapping $\ll$ maps each closed simplex $\s$ of $|\cT(Q)|$ to the geodesic simplex in the sphere with the same vertex set as $\s$. If we restrict to a certain subcomplex $|\cT'(Q)|$ of $|\cT(Q)|$ spanned by the ``minimal'' Schur roots we get a continuous monomorphism onto a dense subset of the sphere. Strictly speaking this is not a triangulation of a subset of the sphere. However, it does express a subset of the sphere as a union of simplices with disjoint interiors.

The $(n-2)$-dimensional faces of our ``triangulation'' are labelled by dimension vectors of indecomposable representations, via semi-invariants, more precisely, supports of semi-invariants of prescribed weights. This labelling depends on the orientation of the quiver. For the vectors with nonnegative coordinates we recover the simplicial complex corresponding to generic decompositions of dimension vectors obtained in \cite{[DW2]}.

In the case of a Dynkin quiver the Complex of Virtual Tilting Sets gives a finite triangulation of the sphere, and it coincides with the Cluster Tilting Triangulation. Its simplices correspond to tilting objects in a corresponding Cluster Category defined in \cite{[BMRRT]}. In addition to the Cluster Tilting Triangulation, we get the labeling of codimension one faces by dimension vectors. This depends on the orientation of the quiver, unlike the Cluster Tilting Triangulation itself. One can show this triangulation is Poincar\'e dual to the Cluster Associahedron of~\cite{[CFZ]}.

In a future paper, we will study the presentation, given by semi-invariants, of the nilpotent
group associated to a Dynkin quiver. This is almost the same as the Steinberg presentation using the Chevalley commutator relations \cite{[St]}. We will also examine the residually nilpotent groups associated to quivers of affine type (also called tame quivers). In prior work \cite{[IO]}, two of the authors constructed an explicit chain resolution for torsion free nilpotent groups and used them to study Milnor's $\overline\mu$-link invariants. We will show in our next paper how this is related to semi-invariants.

Due to our diversity of co-authors we have written this paper to be readable by both topologists and algebraists. 

The following is an outline of the paper. Throughout the paper we assume that $Q$ has no oriented cycles.
In Section~\ref{sec1} we recall the basic definitions and properties of quiver representations and establish notation. In particular we recall the definition of the canonical projective presentation of any representation which Schofield used in his original study of semi-invariants on quiver representations. 
In Section~\ref{sec2} we discuss fundamental known results on semi-invariants of quivers and generic decompositions.  These first two sections serve the reader as background.   
Section~\ref{sec3} defines presentation spaces and their semi-invariants which form directed systems and as such are used to define virtual representation spaces and virtual semi-invariants.  We motivate our definition of virtual semi-invariants in Section~\ref{sec4}
by making  precise, in the case of nonnegative dimension vectors,  the relationship between classical semi-invariants, and semi-invariants on certain special presentation spaces; we show that these rings of semi-invariants are isomorphic.
In Section~\ref{sec5} we prove some properties of presentation spaces and their semi-invariants, including the description of the general elements in presentation spaces.
In Section~\ref{sec6} we introduce the virtual representation spaces. We build on the material of Section~\ref{sec5} to prove  the Stability Theorem which shows that the 
general element in the stable representation space is chain homotopy equivalent to a minimal projective presentation, and thus lies in the orbit of the image of an associated space of minimal presentations.
We extend the Generic Decomposition and the First Fundamental Theorem to virtual dimension vectors.  
In Section~\ref{sec7} we prove the Virtual Saturation Theorem. In Section 7 we construct the simplicial complex of generalized cluster tilting sets and the mapping of the realization of this complex to the sphere. We derive the general properties of this mapping as a consequence of the generic decomposition theorem. In Section 7.2 we also review the definition and properties of the cluster category for comparison. 
Finally in Section~\ref{sec8} we restrict to the special case when the quiver is of Dynkin type and prove the ``Cluster Tilting Triangulation is given by Supports of Semi-invariants'' Theorem \ref{TTDSI}.

The authors thank the referee whose comments significantly improved the exposition of this paper.

\section{Recall: representation spaces}\label{sec1}
This section reviews basic notions related to quiver representations and semi-invariants, and states some of the well known results about semi-invariants from \cite{[S2]}, \cite{[DW1]}.

\subsection{ Quiver representations}\label{Qr}
Let $\k$ be an algebraically closed field.
A quiver $Q$ is a directed graph; denote its set of vertices by $Q_0$ and its set of arrows by $Q_1$. The path algebra, $\k Q$, is the $\k$ algebra generated by the paths in $Q$,  where the product is given by composition of paths, where $\alpha \cdot \beta$ means first traverse the path $\beta$ followed by $\alpha.$  For a given vertex $v \in Q_0,$ the idempotent $e_v$ is the constant path at $v$.  Note that $\k Qe_v$ is the left $\k Q$-module of paths starting at $v$, and similarly for the right $\k Q$ module, $e_u\k Q$.  The algebra $\k Q$ is easily seen to be hereditary. We assume $Q$ is finite and has no oriented cycles, so its path algebra $\k Q$ is finite dimensional.  The category of finitely generated $\k Q$-modules is equivalent to the category of finite dimensional $Q$-representations over $\k$. 

Recall that a {\em quiver representation} is a family of vector spaces $\{M_v\}_{v \in Q_0}$, and linear maps $M_a : M_{ta} \to M_{ha}$ for each arrow $a \in Q_1$, where $ta$ and $ha$ denote the tail and head of $a$, respectively. A map $ f = (f_{v}) $ between two representations $M$ and $M'$ consists of $\k$-linear maps $f_v :M_v\to M_v'$ satisfying commutativity relations: $f_{ha} M_a = M_a' f_{ta}$ for all arrows $a \in Q_1$. (We will sometimes refer to these representations as ``modules'', in order not to confuse them with the generalized or virtual representations, which we will introduce later in \ref{PresSp}, \ref{VirRep}). 
The {\em dimension vector} of $M$ is the vector $\ul\dim M: = (\dim M_v) \in \NN^{n}$ where $n = card(Q_0)$. The {\em radical} of a representation $M$ is $radM$ where $rad \subset \k Q$ is the ideal generated by all arrows, i.e., elements of $Q_1$. 

For a vertex $v$ we denote by $S(v)$ the {\em simple representation} supported at $v$, i.e. $S(v)_v = \k$ and $S(v)_u = 0$ for all $u\ne v$. We also denote by $P(v)$ the {\em canonical indecomposable projective} which maps onto $S(v)$: 
$$P(v)_v=e_v\k,\ \ P(v)_u  = e_u\k Qe_v\otimes_\k P(v)_v\cong \coprod_{\text{paths}\ v\to u}P(v)_v,$$ that is,  the free $\k$-vector space with basis \{all paths from $v$ to $u$\}, and for each arrow $a$, the linear maps $P(v)_a: P(v)_{ta}\to P(v)_{ha}$ are defined on the generating paths by $P(v)_a(p):=ap$. 
Similarly, we denote by $I(v)$ the \emph{canonical indecomposable injective} given by $I(v)_v=\k$, $I(v)_u$ is the free $\k$-vector space with basis \{all paths from $u$ to $v$\}.
More precisely, it is the dual of this vector space, with the dual basis:
\[
	I(v)_u=\Hom_\k(e_v\k Qe_u,I(v)_v)\cong \prod_{\text{paths}\ u\to v}I(v)_v.
\]
An arrow $a:ta\to ha$ gives a linear map
$e_{v}\k Q e_{ha}\to e_{v}\k Qe_{ta}
$
which induces a map $I(v)_{ta}\to I(v)_{ha}$.

Every simple representation is isomorphic to one of the $S(v)$, indecomposable projective to one of the $P(v)$, and indecomposable injective to one of the $I(v)$ (see \cite{[ASS]}, section III.2).

We will use the following notation for projective representations having the prescribed number of indecomposable summands: let $\g \in \NN^{n}$; we denote by $P(\g)$ the following projective representation: $$P(\g)=\coprod_{v\in Q_0} P(v)^{\g_v}.$$
Notice that with this notation, we have $P(\underline{\dim} \,S(v))=P(v)$.

\subsection{ Representation space for $\alpha \in \NN^n$}\label{RepSp}

The representation space for a non-negative integral vector $\a=(\a_v)_{v\in Q_0}$ is the affine space:
$$
R(\a )=\prod_{(u\to v)\in Q_1} \Hom_{\k} (\k^{\a_u}, \k^{ \a_v}).
$$
Elements of $R(\a)$ will be called based representations with {\em dimension vector}
$\a$. Every element of $R(\a)$ can be viewed as a collection of $\a_v \times \a_u$-matrices, one for each arrow $a:$ $u \to v$. The group:
$$
G(\a )=\prod_{v\in Q_0}Gl_{\a_v} (\k)
$$ 
acts on $R(\a)$, by $ g M_a:= (g_{ha})M_a (g_{ta})^{-1}$\label{Gacts}. We use the convention that $Gl_0(\k)$ is the trivial group. Two representations of dimension $\a$ are isomorphic if and only if they lie in the same orbit of the action of $G(\a)$.

\subsection{ Euler matrix and bilinear form}\label{EF}
The vertices of the quiver $Q$ are partially ordered by setting $u<v$ if there is a directed path from $u$ to $v$. Choose a fixed extension of this partial ordering to a total ordering, 
also denoted by $<$.
The {\em Euler matrix}
$E$ is defined as the $n\times n$ matrix with rows and columns labeled by $Q_0$ (order as above), with the diagonal entries equal to 1 and the entry $E_{u,v}=-$(the number of arrows from $u$ to $v$) for $u\ne v$. 
\noindent The {\em Euler form} is the non-symmetric bilinear form on $\ZZ^n$ given by the matrix $E$:
$$
\langle \a ,\b\rangle:=\a^tE\b
.$$

\begin{eg}\label{Example} Quiver $Q$ and the corresponding Euler matrix. The rows and columns of the inverse and transpose of the Euler matrix have interpretations as dimension vectors of projective modules, which is discussed in the next remark \ref{fact}. We use this example to illustrate this.
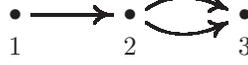
\begin{figure}[htbp]
\begin{center}
{
\setlength{\unitlength}{.6in}
{\mbox{
\begin{picture}(4,1)
     \thicklines
      \qbezier(2.2,.6)(2.5,.8)(2.9,.6)
      \qbezier(2.2,.48)(2.5,.28)(2.9,.48)
      \qbezier(1.2,.55)(1.5,.55)(1.9,.55)
      \qbezier(1.75,.62)(1.8,.55)(1.9,.55)
      \qbezier(1.75,.48)(1.8,.55)(1.9,.55)
      \qbezier(2.75,.58)(2.8,.63)(2.9,.6)
      \qbezier(2.8,.72)(2.8,.66)(2.9,.6)
      \qbezier(2.8,.36)(2.8,.42)(2.9,.48)
      \qbezier(2.75,.49)(2.8,.45)(2.9,.48)
      \put(1,.5){$\bullet$}
      \put(1,.2){1}
      \put(2,.5){$\bullet$}
      \put(2,.2){2}
      \put(3,.5){$\bullet$}
      \put(3,.2){3}
\end{picture}}
}}
\end{center}
\end{figure}
 \vspace{-.2cm}
\[
E =
\left[
\begin{array}{ccc}
  1   &-1  & 0\\
  0& 1 & -2 \\
  0&  0 & 1  
\end{array}
\right],
\  E^{-1} =
\left[
\begin{array}{ccc}
  1   & 1  & 2\\
  0& 1 & 2 \\
  0&  0 & 1  
\end{array}
\right],
\  (E^t)^{-1} =
\left[
\begin{array}{ccc}
  1   &0 & 0\\
 1& 1 & 0 \\
 2&  2 & 1  
\end{array}
\right].
\]
\end{eg}
 \vspace{.1cm}
\begin{rem}[Useful facts about Euler form and Euler matrix]\label{fact}
(\cite{[ASS]}, sec. III.3).
\begin{enumerate}
\item 
$\langle \a,\a' \rangle =\dim_{\k}\Hom_{Q}(M,M') -\dim_{\k}\Ext^1_{Q}(M,M')$ for all representations $M$ and $M'$ such that $\a =\underline{\dim}\,M$ and  $\a' =\underline{\dim}\,M'$.
\item The row of $E$ corresponding to the vertex $v$, consists of coefficients of the dimension vector $\ul\dim\, S(v)$ written as linear combination of dimension vectors $\ul\dim\, P(w)$. In the example above, we have 
$\ul\dim\, S(2) = \ul\dim\, P(2) - 2\ul\dim\, P(3).$
\item $E^t \ul\dim \,P(v) = \ul\dim \,S(v)$. Equivalently, 
$(E^t)^{-1} \ul\dim\, S(v) = \ul\dim\, P(v)$. In other words, $\ul\dim\, P(v)$ is the column of $(E^t)^{-1}$ corresponding to vertex $v$.
\item The product $E^t \a$ gives the coefficients of a vector  $\a$ written as a linear combination of the vectors $\underline{\dim}\,P(w)$. 
In particular, $E^t \ul\dim\, P(\g) = \g$ for any $\g \in \NN^n$.
\item  
$(E^t)^{-1}(\g) = \ul\dim\, P(\g)$ for all $\g \in \NN ^n$.  In particular, for $\g \in \NN^n$,   $(E^t)^{-1}(\g) \in \NN^n$.
\item $\a - E^t(\a)$  has non-negative coefficients for $\a\in\NN^n$, e.g. if $\alpha = \underline{\dim}\,S(v)$ then 
$\a - E^t(\a)= \underline\dim\,(radP(v)/rad^2P(v))$. 
\item The dimension vectors $\ul\dim\, I(v)$ of the indecomposable injective vectors occupy the columns of $E^{-1}$. In particular, the entries of $E^{-1}$ are all non-negative.
\end{enumerate}
\end{rem}

\subsection{ Canonical projective presentations}\label{CPP}
Recall, the {\em canonical projective presentation} of a representation $M$, with $\underline{\dim}M = \a$ is:
\[0\to P_1\xrarrow{p_{M}} P_0\to M\to0\]
$$\text{where}\quad P_1=\coprod_{(u\to v)\in Q_1}P(v)^{\a_u} \ \ \ \text{ and } \ \ \ \  P_0=P(\a)=\coprod_{v\in Q_0}{P(v)^{\a_v}}.$$ 
The mapping $p_{M}$ can be described as follows: For each arrow $a=(u\to v)$ $ \in Q_1$, the restriction of $p_M$ to $P(v)^{\alpha_u}$ is given by:
$$P(v)^{\alpha_u}\xrarrow{((- incl_a)^{\alpha_u},M_a)} P(u)^{\alpha_u}\smallcoprod P(v)^{\alpha_v},$$
where $M_a: M_{ta}\to M_{ha}$ is the linear map associated to the arrow $a$ in the definition of the representation $M$ (as in \ref{Qr}), and $incl_a: P(v) \to P(u)$ is the inclusion map corresponding to the arrow $a: u\to v$. The representation $P_1$ can also be rewritten as:
$$
P_1=\coprod_{(u\to v)\in Q_1}P(v)^{\a_u} =\coprod_{v\in Q_0} P(v)^{(\sum_{u\to v}\a_u)}=P(\a-E^t\a).
$$ 
 Consequently, $p_M \in \Hom_Q(P(\a-E^t\a),P(\a))$.
 
One constructs the canonical injective resolution similarly.


\section{Recall: classical results on semi-invariants of quivers}\label{sec2}
We recall now the notion of semi-invariants of a group acting on a variety and state the classical results about semi-invariants on the representation spaces of quivers by Kac, Schofield and Derksen-Weyman. 

\subsection{ Definition of semi-invariants}\label{DefSI}
For an algebraic group $G$ acting on a variety $X$, an element $f$ of the coordinate ring of $X$ is called a {\em semi-invariant}, if there exists a character  $\chi $ of $G$ such that for all $g\in G$ and all $v\in X$:
$$
f(g\cdot v) = \chi(g)f (v).
$$
We will refer to $\chi$ as \emph {the character of the semi-invariant} $f$.
\begin{rem}
The rational characters (characters which are rational functions) on $Gl_n(\k)$ are $\det(g)^s$ where $s\in\ZZ$, providing that $\k$ has at least 3 elements and $n\geq1$.
\end{rem}

\subsection{ Semi-invariants of quivers}\label{SIQ} Let $Q$ be a quiver with $n$ vertices and let $\a =(\a_1,\dots,\a_n) \in \NN^n$.
The group $G(\a)$ acts on the representation space $R(\a)$ (as described in \ref{RepSp}). 
Since the group $G(\a)$ is the product of general linear groups, the character $\chi$ at $g$ is the product $\chi(g)=(\det(g_1))^{\s_1}\dots (\det(g_n))^{\s_n}$, where $\s =(\s_1,\dots,\s_n)\in \ZZ^n$.  The following are some basic facts about semi-invariants, their characters and weights.

\begin{defn}Let $Q$ be a quiver, $n=|Q_0|$ and $\a\in\NN^n$.
Let $f$ be a semi-invariant on $R(\a)$ and $\chi$ the uniquely determined character of $f$.
\emph{A weight of the semi-invariant} $f$ is any  vector $\sigma= (\sigma_1, \dots ,\sigma_n)$, for which the character $\chi$ for $f$ can be written as 
$$\chi(g) = (\det(g_1))^{\s_1}\dots (\det(g_n))^{\s_n}.$$
Furthermore, a vector $\s\in\ZZ^n$ will be called a \emph{weight} if it is a weight for some semi-invariant.
\end{defn}

\begin{rem}  Let $\a\in\NN^n$ be fixed.
\begin{enumerate}
\item Each vector $\s\in \ZZ^n$ determines a unique character, which we denote by $\chi_{\s}$.  
\item On the other hand, a character $\chi$ for some semi-invariant, might not uniquely determine the weight of the semi-invariant. 
This happens in the important non-sincere case: if $\alpha_i = 0$, then $g_i$ is a $0\times0$ matrix, in which case $\det(g_i) = 1$, therefore for any $\s_i \in \mathbb Z$, $\det(g_i)^{\s_i} = \det(g_i) = 1$.
\item A character $\chi$ for a semi-invariant $f$ determines uniquely a coset of a free abelian subgroup of $\ZZ^n$. The coset consists of all weights of $f$,\\
 $\Sigma_{\chi}=\Sigma_{f} =\{\s\ |\ \s \text{ is a weight of }f \}.$
\end{enumerate}
\end{rem}

\begin{defn}\label{ConeW} Let $\a\in\NN^n$. Define, as in \cite{[DW1]} \emph{the cone of weights} as:
 $$\Sigma(Q,\a):=\{\s\ |\ \s \text{ is a weight of some semi-invariant on } R(\a)\}.$$
\end{defn}

We now define rings of semi-invariants as graded rings, where the grading is given by the characters of the group $G(\a)$.
\begin{defn}
For $\alpha\in \NN^n$, we denote by $ SI(Q,\a )_{\chi}$ the $\k$-vector space of semi-invariants on $R(\a)$ with the character $\chi$ and by $SI(Q,\a )$ the graded ring
$$
SI(Q,\a )=\bigoplus_{\chi\in \text{Char} G(\a)} SI(Q,\a )_{\chi},
$$
called the {\em ring of semi-invariants} for the action of $G(\alpha)$ on $R(\alpha)$. 
\end{defn}

\begin{rem}\label{conventions} We point out the following facts and conventions:
\begin{enumerate}
\item  Our convention differs from \cite{[DW1]} in that the weights are negated.
\item The polynomial $0$ appears as a semi-invariant in each grading, i.e. of all possible weights.
\end{enumerate}
\end{rem}

\subsection{ Fundamental theorems for semi-invariants of quivers}\label{fundthm} The \emph{First Fundamental Theorem} (FFT) states that rings of semi-invariants of quivers are spanned by determinants.  The \emph{Saturation Theorem} describes all nonnegative vectors with semi-invariants of a given weight. The third theorem, the \emph{Generic Decomposition Theorem} (Kac terminology), describes the decomposition of a general representation of any non-negative integral vector  $\a \in \NN^n$.

We recall the definitions of general representations and also of the fundamentally important polynomial functions $c_V$.

\begin{defn} \label{GenRep} Let $Q$ be a quiver and $\k$ a field. 
A \emph{generic representation} of $Q$ of dimension $\a$ is a representation over a transcendental extension (i.e., field of rational functions) of $\k$ with one variable for each entry of the matrix representation of the representation.  Alternatively stated, \emph{a general representation} is a representation from some nonempty Zariski open set in $R(\alpha )$.
\end{defn}

Every semi-invariant vanishing on a generic (or general representation) is identically zero.

\begin{defn}
Let $Q$ be a quiver and let $\a\in\NN^n$.  For a representation $V$ of $Q$, define 
$
(p_M, V) : \Hom{(P_0, V)} \to \Hom{(P_1, V)}
$
to be the vector space homomorphism induced by the canonical presentation $p_M$ of $M \in R(\alpha)$, as defined in \ref{CPP}.   Given $V$ such the $(p_M, V)$ is square (that is, by Useful fact~\ref{fact}.2, such that $\<\a,\ul {\dim} V\>=\a^t E\ul {\dim}V=0$), define the polynomial function $c_V \in \k [R(\a  )]$ by setting
$$
c_V (M):= \det(p_M,V).
$$
\end{defn}
Note that a decomposition $V=V_1\smallcoprod V_2$ gives a factorization of $c_V(M)$ as $c_V(M) = c_{V_1}(M) c_{V_2}(M)$, providing $\<\a,\ul {\dim} V_1\>=\<\a,\ul {\dim} V_2\>=0.$

\begin{thm}[FFT,\cite{[S2],[DW1]}, see Remark~\ref{conventions}]\label{thm: FFT}
Let $Q$ be a quiver and $\a \in \NN^n$. Then the ring of semi-invariants $SI(Q,\a )$ is spanned as a $\k$-vector space by the functions
$c_V$ for representations $V$ satisfying $\<\a,\ul {\dim} V\>=\a^t E\ul {\dim}V=0$. Furthermore, the character of the semi-invariant $c_V$ is $\chi_{\sigma}$ where $\sigma = E \ul{\dim}V$.
\end{thm}

By simply restricting to the support of $\a$ we get the following.

\begin{cor}\label{cor to FFT}
$SI(Q,\a)$ is spanned by those $c_V$ as above where $V_v=0$ whenever $\a_v=0$. 
\end{cor}
%

\begin{defn}\label{Supp} Let $\s\in\ZZ^n$. The $\NN$-{\em support} of  $\s$ is defined as 
$$
supp_{\NN}(\s):= \{\a\in\NN^n\ |\ SI(Q,\a )_{\chi_{\sigma}} \ne 0\}.
$$
\end{defn}
\begin{rem}  It follows that $\a \in supp_{\NN}(\s)$ if and only if $\s \in \Sigma (Q,\a)$, as in the reciprocity theorem of \cite{[DW1]}, or definition \ref{ConeW}.
\end{rem}

The next theorem is clearly equivalent to the Saturation Lemma in \cite{[DW1]} in view of the Reciprocity Property of that paper. In order to state the theorem we need to define the sets $D(\b)$ which will be used throughout the paper.

\begin{defn} \label{DefDb} Let $\b\in\NN^n$. Define the subset $D(\b)\subset \RR^n$ as
$$
D(\b):=\{\a\in\RR^n\st \<\a,\b\>=0\}\cap(\cap_{\b'\into\b}\{\a\in\RR^n\st \<\a,\b'\>\leq0\});$$
here $\b'\into\b$ means that the general representation of dimension $\b$ has a subrepresentation of dimension $\b'$.

\end{defn}

\begin{thm}[Saturation, \cite{[DW1]}]\label{satthm}
Let $\b\in\NN^n$. Then: 
$$
supp_{\NN}(E\b)=\NN^n\cap D(\b).
$$
\end{thm}

{Before stating the Generic Decomposition Theorem recall the definition of Schur root, and also of $hom$ and $ext$ on vectors in $\NN^n$.

\begin{defn} \label{Schur} Let $Q$ be a quiver and $\a\in\NN^n$. Then $\a$ is called a \emph{Schur root} if the general representation in $R(\a)$ is indecomposable. 
\end{defn}

\begin{defn}
Let $Q$ be a quiver and $\a,\b\in\NN^n$. Define: 
$$
hom_Q(\a,\b):=\min\{\dim_{\k}\Hom_Q (A,B)|\  \ul{\dim}A = \a, \ul{\dim}B=\b\}.
$$
$$
ext_Q(\a,\b):=\min\{\dim_{\k}\Ext_Q (A,B)|\  \ul{\dim}A = \a, \ul{\dim}B=\b\}.
$$
\end{defn}
Since $\dim_\k\Hom$ and $\dim_\k\Ext$ are upper semicontinuous and $\k$ algebraically closed, these minima are attained for general modules of these dimension vectors. So, this definition agrees with the usual definition, i.e. $hom_Q (\a,\b) = \dim_{\k}\Hom_Q(A,B)$ and $ext_Q (\a,\b) = \dim_{\k}\Ext_Q(A,B)$, where $A,B$ are general representations with $\ul{\dim}A = \a$ and $\ul{\dim}B = \b$.

\begin{thm}[Generic Decomposition,\cite{[DW2]}]\label{classical generic decomp} Any $\a \in \NN^n$ has a unique decomposition of the form
$\a = \Sigma \a_i$
where $ext_Q (\a_i,\a_j)=0$ for all $i\ne j$ and each $\a_i$ is a Schur root. Furthermore, the general representation $M$ with $\ul{\dim} M = \a$ decomposes as $M\cong \coprod M_i$ with $\ul\dim M_i=\a_i$ where $M_i$ are indecomposable representations so that $\Ext_Q(M_i,M_j)=0$ for all $i\neq j$.
\end{thm}

\section{Define: presentation spaces and their semi-invariants}\label{sec3}
In this section we deal with integral vectors (not necessarily non-negative), define presentation spaces associated to these vectors, and consider semi-invariants with respect to the actions of certain non-reductive algebraic groups. In order to justify this, we prove in the next section \ref{iso of rings of sis}, that for non-negative vectors the rings of semi-invariants on certain special presentation spaces are isomorphic to the classical rings of semi-invariants on the quivers as in \cite{[DW1]}. In later sections we will prove analogous theorems to the three fundamental theorems.

\subsection{ Projective decompositions of integral vectors}\label{PD}
Let $\a\in\ZZ^n$ and let 
$$E^t\a=\g_0-\g_1 \text{ with }\g_0, \g_1\in\NN^n.$$ We refer to $(\g_0,\g_1)$  as a {\em projective decomposition} of $\a$ since by \ref{fact}(5) we have $\a = \ul\dim P(\g_0)-\ul\dim P(\g_1)$.
The set of projective decompositions $(\g_0 ,\g_1 )$ of $\a$
forms a directed partially ordered set $PD(\a )$ with partial ordering given by: $$(\gamma_0,\gamma_1)\leq(\gamma'_0,\gamma'_1) \ \text{ if }\  (\gamma'_0,\gamma'_1)=(\gamma_0 + \gamma, 
\gamma_1 + \gamma) \text { for some } \gamma \in \NN^n.$$ 
 Note that there is a unique {\em minimal projective decomposition} where $\g_0,\g_1$ have disjoint supports; (with $\g_0$ being the positive and $-\g_1$ the negative part of $E^t\a$.) 
 
\subsection{ Presentation spaces}\label{PresSp}
Let $\alpha \in \ZZ^n$. 
 For each projective decomposition $(\gamma_0,\gamma_1)$ of $\a$, we define a 
 {\em presentation space}
$$R
(\g_0 ,\g_1): = \Hom_Q(P(\g_1),P(\g_0)).
$$
Definitions and references for some of the special presentation spaces we use in this paper:
\begin{itemize}
\item \emph{Minimal presentation space }$R^{min}(\a):= R(\g_0,\g_1)$ for $\a\in \ZZ^n$ where $(\g_0,\g_1)$ is minimal projective decomposition. (See the Stability theorem \ref{thm:stability theorem}).
\item\emph{Canonical presentation space} $R^{can}(\a):= R(\b, \b - E^t\b + \g)$ for $\a\in \ZZ^n$ will be precisely defined in \ref{Rcan}.
\item The special case $R(\a, \a- E^t\a)$ for non-negative $\a\in \NN^n$ is particularly important for several reasons. The canonical projective presentation is an element of it (see \ref{CPP}).  Also, we will show that this is a special case of the canonical presentation space $R^{can}(\a)$, which is introduced in the subsection \ref{subsec canonical presentation space} and is important for the virtual generic decomposition theorem \ref{VirGenDec}.

\end{itemize}

The space 
$R(\g_0 ,\g_1) =\Hom_Q(P(\g_1),P(\g_0))$ 
is an affine space with the natural action of the group 
$\Aut(P(\g_0))\times (\Aut(P(\g_1)))^{op}$
which is given by 
$$(g^0, g^1)\varphi: = g^0 \varphi g^1$$
for $(g^0,g^1)\in \Aut(P(\g_0))\times \Aut(P(\g_1))^{op}$ and for each $\varphi \in \Hom_Q(P(\g_1),P(\g_0))$.  
 
 \begin{defn}\label{IsoPres}
Two elements of the presentation space $R(\g_0,\g_1)$ are called \emph{isomorphic} if they lie in the same orbit of the action of this group of automorphisms.
\end{defn}
\begin{rem} \label{RankGenPr} Similarly to general representations as in Definition \ref{GenRep}, we have:\\
(1) The \emph{general presentation} or \emph{general element} 
of the presentation space $R(\g_0,\g_1)$ is any element of a nonempty Zariski open subset of $R(\g_0,\g_1)$.\\
(2) The rank of the general element in $R(\g_0,\g_1)$, i.e. general presentation\\ $P(\g_1)\xrarrow{\f} P(\g_0)$  is the maximum of all ranks of all presentations in $R(\g_0,\g_1)$.
\end{rem}

\subsection{ Semi-invariants on presentation spaces}\label{Char}
Since $R(\g_0 ,\g_1)$ is an affine space, its coordinate ring is a polynomial ring and we consider the semi-invariants for the action of $\Aut(P(\g_0))\times (\Aut(P(\g_1)))^{op}$ on the coordinate ring $\k [R(\g_0 ,\g_1)]$. 
Recall that a semi-invariant on $R(\g_0 ,\g_1) =\Hom_Q(P(\g_1),P(\g_0))$ is a polynomial function $f$ such that for some character $\chi$: 
$$
f((g^0,g^1)\varphi)=\chi(g^0, g^1)f(\varphi)
$$
for all $(g^0,g^1)\in \Aut(P(\g_0))\times (\Aut(P(\g_1)))^{op}$ and all $\varphi\in \Hom_Q(P(\g_1),P(\g_0))$. 

 \begin{prop} \label{Prop(1-4)}
Some facts about the characters of the semi-invariants on the presentation spaces $R(\g_0,\g_1)$:
\begin{enumerate}
 \item Since the group is a product of two groups, we have: 
$$
\chi(g^0,g^1) = \chi^{0}(g^0) \chi^{1}(g^1),
$$ 
with $\chi^0$ and $\chi^1$ characters of $\Aut(P(\g_0))$ and $(\Aut(P(\g_1)))^{op}$ respectively.

\item Notice that for $P(\g)=\coprod_v{P(v)^{\g_v}}$, each element of $\Aut(P(\g))$ can be written as an $n\times n$ block triangular matrix $g=(g_{uv})$ with 
$$
g_{uv}\in\Hom_Q({P(v)}^{\g_v},{P(u)}^{\g_u})
\ \ \ 
\text{and}
\ \ \ 
g_{vv}\in\Aut({P(v)}^{\g_v})\cong Gl_{\g_v}(\k).
$$
Using the last isomorphism we identify these two groups and also the groups
$$\prod_v \Aut (P(v)^{\gamma_v})\ \text{ and }\ G(\gamma)=\prod_v Gl_{\gamma_v}(\k).$$\label{AutG}
We use the total order $u<v$ defined in section \ref{EF}  to write the matrix.
 
\item For a projective $P(\g)$, any character of the group $\Aut(P(\g))$ has the form:
$$\chi_{\s}(g)=\prod_{v\in Q_0}\det(g_{vv})^{\s_v},$$ 
with the weight vector $\s\in \ZZ^n$; however for the semi-invariants on presentation spaces, $\s\in\NN^n$. As in \ref{SIQ}, if $\g_v=0$, i.e. $g_{vv}\in Gl_0(\k)$ then $\det(g_{vv})=1$  and  $\sigma_{v}$ is indeterminate. 
\label{Indet}
\end{enumerate}
\end{prop}

With the above, we see that the semi-invariants on presentation spaces have pairs of characters associated to them and also pairs of weights.

\begin{defn} \label{DefPsSym} Denote  by $SI^{(\g_0 ,\g_1 )}(Q,\a )_{(\chi^0 ,\chi^1)}$ the  set of semi-invariants on $R(\g_0 ,\g_1)$ with character $(\chi^0 ,\chi^1)$  and the associated graded ring of semi-invariants by: 
$$SI^{(\g_0 ,\g_1 )}(Q,\a ): = \bigoplus_{(\chi^0 ,\chi^1)} SI^{(\g_0 ,\g_1 )}(Q,\a )_{(\chi^0 ,\chi^1)}.$$
\end{defn}

\begin{prop}\label{all si are p-symmetric}
Let $\a\in\ZZ^n$, let $(\g_0,\g_1)$ be a projective decomposition of $\a$, and $R(\g_0,\g_1)$ the corresponding presentation space. 
Let $f$ be a semi-invariant on $R(\g_0,\g_1)$ with the character $(\chi^0_{\s^0},\chi^1_{\s^1})$. Then $\s^0_v=\s^1_v$ if both $\g_{0,v}\neq0$ and $\g_{1,v}\neq0$.
\end{prop}

\begin{proof} 
Let $f$ be a semi-invariant on $R(\g_0,\g_1)$, $f((g^0,g^1)\varphi) = \chi^0_{\s^0}(g^0) \chi^1_{\s^1}(g^1) f(\varphi)$ for all 
$(g^0,g^1)\in \Aut P(\g^0)\times \Aut P(\g^1)^{op}$ and all $\varphi \in R(\g^0,\g^1)$. We need to show that $\s^0_v=\s^1_v$ for all $v\in Q_0$ for which both $\g_{0,v}\neq 0$ and $\g_{1,v}\neq 0$.

If $(\g_0,\g_1)$ is the minimal projective decomposition of $\a$ then there is no $v\in Q_0$ such that both $\g_{0,v}\neq 0$ and $\g_{1,v}\neq 0$, so there is nothing to prove.

In order to deal with any projective decomposition, define  $\g_v:=min\{\g_{0,v},\g_{1,v}\}$. Then $\g_v\neq 0$ precisely at the vertices where both $\g_{0,v}\neq 0$ and $\g_{1,v}\neq 0$. 
Let $\g\in\NN^n$ be defined as $\g= (\g_v)$. Then $(\g_0-\g),(\g_1-\g)\in \NN^n$, actually
$(\g_0-\g,\g_1-\g)$ is the minimal projective decomposition of $\a$.

Since we need to check the weights of the semi-invariant $f$ only at the vertices $v$, where $\g_v\neq 0$, we will consider the following presentations and group elements
$$\varphi = \varphi'\smallcoprod 1_{P(\g)}\in \Hom_Q\big(P((\g_1-\g)\smallcoprod P(\g),P(\g_0-\g)\smallcoprod P(\g)\big) \ \text{ and }$$ 
$$(1_{P((\g_1-\g)}\smallcoprod g)\times (1_{P((\g_0-\g)}\smallcoprod g^{-1})\in \Aut(P(\g_0))\times  \Aut(P(\g_1))^{op}$$ 
 for $g\in \Aut(P(\g))$. Then we have 
 $$f(\varphi'\smallcoprod 1_{P(\g)})=
f\left((1_{P((\g_0-\g)}\smallcoprod g)\cdot (\varphi'\smallcoprod 1_{P(\g)})\cdot (1_{P((\g_1-\g)}\smallcoprod g^{-1})\right)=$$
$$\chi_{\s^0}^0(g)\chi_{\s^1}^1(g^{-1})f(\varphi'\smallcoprod1_{P(\g)}).
$$
So, $\chi_{\s^0}^0(g)=\chi_{\s^1}^1(g)$. Therefore $\s^0_v=\s^1_v$ for all $v$ for which $\g_v\neq0$.
\end{proof}

\begin{defn}
Let $f$ be a semi-invariant on $R(\g_0 ,\g_1)$. 
The \emph{combined weight} $\s = \s^{comb}$ of $f$ is defined to be  $\s^{comb}_{v} : = max\{\s^{0}_{v}, \s^{1}_{v}\}$ for all $v\in Q_0$,
and the \emph{combined  character} $\chi_{\s}$ to be 
$(\chi_{\s},\chi_{\s})$.
\end{defn}

\begin{defn} \label{DefPsSym combined} Denote  by $SI^{(\g_0 ,\g_1 )}(Q,\a )_{\chi}$ the  set of semi-invariants on $R(\g_0 ,\g_1)$ with combined character $\chi$.
\end{defn}

\section{Relate: representation and presentation spaces and their semi-invariants for $\a\in \NN^n$}\label{sec4}

Now consider only non-negative integral vectors $\a\in\NN^n$ and compare the classical representation space $R(\a)$ and the special presentation space $R(\a,\a-E^t\a)$ together with the natural group actions of $Gl(\a)$ and $\Aut P(\a) \times \Aut P(\a-E^t\a)^{op}$. We give relations between these spaces and prove that their rings of semi-invariants are isomorphic (\ref{iso of rings of sis}).

\subsection{ Relations between representation and presentation spaces for $\a\in \NN^n$}\label{subsec4}

In order to compare these two spaces, we define the mapping
$$\z: R(\a)\to R(\a,\a-E^t\a)$$ 
as
$\z(M)= p_M$, where $p_M$ is the canonical projective presentation of $M$ (\ref{CPP}).
Consider the subspace $Im \z \subset R(\a,\a-E^t\a)$ and orbits of this subspace under the action of the groups $\Aut P(\a)$, $\Aut P(\a-E^t\a)^{op}$ and $\Aut P(\a) \times \Aut P(\a-E^t\a)^{op}$.

For each projective module $P(\a)$ define $T(\a)$ in the following way: 
if $\a=e_v$, the unit vector at $v$, then $P(\a):=P(v)$ as in \ref{Qr}, i.e. it is a vector space generated by all paths starting at $v$. Let $T(e_v):=\k e_v$ be the linear subspace generated by the constant path at $v$. For any $\a\in\NN$, we have a decomposition of $\a$ as a sum of unit vectors $e_v$. This way we have chosen an internal direct sum decomposition $P(\a) =\sum P(e_v)$. Let $T(\a)=\sum T(e_v)$.


\begin{defn} Let $U(\a,\a-E^t\a)\subset R(\a,\a-E^t\a)$ be the open subspace
defined as:
\noi$\{\psi:P(\a-E^t\a)\to P(\a) |\ \psi \text{ monomorphism, } Im(\psi) \text{ is complementary to } T(\a)\}.$
\end{defn}

\begin{eg}
For example, if $P(\a)$ is indecomposable, then $rad P(\a)$ is a submodule of $P(\a)$ which is complementary to $T(\a)$ since $P(\a)=T(\a)\oplus rad P(\a)$ as a vector space over $\k$; however, when $P(\a)$ is not indecomposable there might be other such submodules which are not equal to $rad P(\a)$, although each of them is isomorphic to $rad P(\a)$ as a representation.
\end{eg}
\begin{lem}\label{lem} 
Let $\a\in\NN^n$ and $\z: R(\a)\to R(\a,\a-E^t\a)$ be defined as $\z(M)= p_M$, the canonical projective presentation. Then the orbit of $Im(\z)$ under the action of $\Aut P(\a-E^t\a)^{op}$ is $U(\a,\a-E^t\a)$ and is thus open.
\end{lem}

\begin{proof} 
We show that $\Aut(P(\a-E^t\a))^{op} Im(\z)=Im(\z)\Aut P(\a-E^t\a)$ is equal to $U(\a,\a-E^t\a)$ and is thus open. Let $\elemofU:P(\a-E^t\a)\to P(\a)$ be an element of $U(\a,\a-E^t\a)$ and let $M=\coker \elemofU$. Then, by definition, the quotient map $P(\a)\to M$ is the same as the map $\maptoM$ in the canonical projective presentation
\[
	P(\a-E^t\a)\xrarrow{\z(M)}P(\a)\xrarrow{\maptoM}M.
\]
Therefore the image of $\z(M)$ is the same as the image of $\elemofU$, and $\elemofU$ and $\z(M)$ differ by an automorphism of $P(\a-E^t\a)$.
\end{proof}

\begin{prop}\label{orbit prop} 
Let $\a\in\NN^n$ and $\z: R(\a)\to R(\a,\a-E^t\a)$, $\z(M)= p_M$  the canonical projective presentation. Then
the orbit of $Im(\z)$ under the action of $\Aut P(\a)\times\Aut P(\a-E^t\a)^{op}$ is an open and dense subset of $R(\a,\a-E^t\a)$.
\end{prop}

\begin{proof}
Since $\Aut P(\a-E^t\a)^{op}$ is a subgroup of $\Aut P(\a)\times \Aut P(\a-E^t\a)^{op}$ and the $\Aut P(\a-E^t\a)^{op}$ orbit of $Im(\z)$ is open in $R(\a,\a-E^t\a)$ the result follows. 
\end{proof}

\begin{rem}\label{rem:general properties of reps and maps}
General properties of representations (properties that hold on an open subset of $R(\a)$) are also general properties of elements of $R(\a,\a-E^t\a)$. This proposition tells us that, conversely, the general intrinsic (i.e. invariant under isomorphism) properties of elements of $R(\a,\a-E^t\a)$ are also general properties of elements of $R(\a)$.
\end{rem}

\begin{lem}
There is a 1-1 correspondence, given by quotients, between the submodules of $P(\a)$ which are complementary to $T(\a)$ and the elements of the representation space $R(\a)$. Furthermore, all such submodules are isomorphic to $P(\a-E^t\a)$.
\end{lem}
\begin{proof}
Given a submodule $L\subset P(\a)$ which is complementary to $T(\a)$, we take the quotient module $P(\a)/L$. Since this is vector space isomorphic to $T(\a)$, the structure maps are matrices and we get an explicit element of $R(\a)$. 

Given any $M\in R(\a)$ the corresponding submodule of $P(\a)$ is the kernel of the canonical projection map $\maptoM:P(\a)\to M$. This is also the image of the canonical presentation map $p_M:P(\a-E^t\a)\to P(\a)$ which is always a monomorphism with image complementary to $T(\a)$.

These constructions are clearly inverse to each other.
\end{proof}

\begin{prop}\label{coker prop} 
Cokernels of homomorphisms define a mapping, which we denote by $
	\coker:U(\a,\a-E^t\a)\to R(\a)$. Furthermore,
\begin{enumerate}
\item $\coker$ is a rational map.\label{rationalmap}
\item $\coker\circ \z = Id_{R(\a)}$ hence $\z$ is a monomorphism.
\end{enumerate}
\end{prop}

\begin{proof}
For each $\elemofU\in U(\a,\a-E^t\a)$, the representation $L=\im \elemofU$ is complementary to $T(\a)$ by definition. Therefore $\coker\elemofU=P(\a)/L$ is an element of $R(\a)$ by the above lemma. Straightforward linear algebra shows that this is a rational map. The canonical presentation of any element of $R(\a)$ lies in $U(\a,\a-E^t\a)$ and the following composition is the identity map:
$
	R(\a)\xrarrow{\z}U(\a,\a-E^t\a)\xrarrow{\coker}R(\a).
$
\end{proof}

\subsection{ Semi-invarints on representation and presentation spaces for $\a\in \NN^n$}

First we show that the weights of semi-invariants on the classical representation space and the new presentation space are related by the Euler matrix. Then we use \ref{orbit prop} to show that the ring of semi-invariants on $R(\a,\a-E^t\a)$ is generated by the $\cc_V$'s which are the classical $c_V$'s, i.e., $\det\Hom_Q(p_M,V)$ but now evaluated on all elements of $R(\a,\a-E^t\a)$.



\begin{prop}\label{wrelation}
Let $\a\in\NN^n$ and $f$ be a semi-invariant on $R(\a,\a-E^t\a)$ with combined character $\chi_{\s}$. Then $f\circ\z$ is a semi-invariant on $R(\a)$ with character $\chi_{E\s}$.
\end{prop}

\begin{proof} 
By assumption on $f$ we know that
$f((g^0,g^1)p) = \chi_{\s}(g^0)\chi_{\s}(g^1) f(p)$
for all $(g^0,g^1)\in \Aut(P(\a))\times \Aut (P(\a-E^t\a))^{op}$, all 
$p\in R(\a,\a-E^t\a)= \Hom_Q(P(\a-E^t\a),P(\a))$ and
 for some combined character $\chi_{\s}$.
We want to show
$$(f\circ\z)(gM) = \chi_{E\s}(g) (f\circ \z)(M)$$
for all  $g=(g_v)\in Gl(\a)=\prod Gl_{\a_v}(\k)$, all 
$M\in R(\a)$ and
 for character $\chi_{E\s}$ .
 
 By definition the representation $gM$ consists of vector spaces $(gM)_v=M_v$ for all $v\in Q_0$ and $(gM)_{u\to v}= g_v\circ M_{u\to v} \circ g^{-1}_u$ for all $(u\to v) \in Q_1$ (\ref{Gacts}).

Then $\z(gM) = p_{gM}$, the canonical projective presentation of $gM$ fits in the following commutative diagram of $Q$-representations:

\[\begin{CD}
	P(\a-E^t\a)=\coprod_{u\to v}P(v)^{\a_u} @>\z(M)=p_M>> P(\a)=\coprod_v P(v)^{\a_v} @>>> M@>>> 0\\
	@VV{\varphi_1(g)}V @VV{\varphi_0(g)}V @VVgV\\
	P(\a-E^t\a)=\coprod_{u\to v} P(v)^{\a_u} @>\z(gM)=p_{gM}>>P(\a)= \coprod_v P(v)^{\a_v} @>>> gM@>>>0,
\end{CD}
\]\\
where $\varphi_0$ and $\varphi_1$ are defined in the following way: after identification 
$\prod_v Gl_{\a_v}(\k) = \prod_v \Aut(P(v)^{\a_v})$  as stated in Proposition \ref{Prop(1-4)}(3),  $\varphi_0(g) = g$ and $\varphi_1(g) = h$, where $h_w=\prod_{v\to w}g_v$.

Now, it follows from the above diagram that
\begin{equation}\label{equivariance of zeta}
(f\circ \z)(gM) = f(\z(gM)) = \varphi_0(g)\circ \z(M) \circ (\varphi_1(g))^{-1}.
\end{equation}
By the definition of the group action on $R(\a, \a-E^t\a)$ this equals
$$ f((\varphi_0(g),(\varphi_1(g))^{-1})\z(M)).$$ 
By the assumption that $f$ is semi-invariant of combined character $\chi_{\s}$ this equals
$$ \chi_{\s}(\varphi_0(g))\chi_{\s}(\varphi_1(g))^{-1} f(\z(M)).$$
Finally, by the definitions of $\varphi_0$ and $\varphi_1$ and the fact that these characters are given by determinants this equals
$$ \prod_v \det(g_v)^{\s_v} \cdot \prod_{w} \det(h_w)^{-\s_w} \cdot f(\z(M)) =$$
$$= \prod_v \det(g_v)^{\s_v} \cdot \prod_{v\to w}\det(g_v)^{-\s_w} \cdot f(\z(M)) =$$
$$= \prod_v \det(g_v)^{\s_v-\Sigma_{v\to w}\s_w} (f\circ\z)(M) =\prod_v \det(g_v)^{(E\s)_v} (f\circ\z)(M)=$$
$$= \chi_{E\s}(g)(f\circ\z)(M).$$

Thus, $f\circ\z$ is a semi-invariant on $R(\a)$ with character $\chi_{E\s}$.
\end{proof}
\begin{rem} \label{Indeterm} We note that $E\s$ may not determine $\s$ even though $E$ is invertible. The reason is that the weight $E\s$ of $f\cdot \z$ may have more indeterminacy then the weight $\s$ of $f$. $E\s$ and $\s$ have the same indeterminacy, i.e. $\s\to E\s$ maps  the weight coset of $f$ onto the weight coset of $f\cdot\z$, if and only if the support of $\a-E^t\a$ is contained in the support of $\a$.

\end{rem}

First we define the maps $\cc_V:R(\a,\a-E^t\a)\to \k$ which extend the semi-invariants $c_V:R(\a)\to\k$.

\begin{defn}
Let $\a\in\NN^n$  and $V$ be a representation such that $\<\a,\ul\dim V\>=0$. Define $\cc_V(\psi):=\det\Hom_Q(\psi,V)$ for $\psi\in R(\a,\a-E^t\a)$.
\end{defn}
\begin{rem}\label{rem:zeta sends Cv to cv}
Notice that $\cc_V(\z M)=\cc_V(p_M)= \det\Hom_Q(p_M,V)$ which is equal to $c_V(M)$ by definition of $c_V$. In other words, the composition
\[
	R(\a)\xrarrow{\z} R(\a,\a-E^t\a)\xrarrow{\cc_V} \k
\]
coincides with the classical semi-invariant $c_V$ on $R(\a)$ as in \ref {thm: FFT}.
 \end{rem}
 
\begin{lem} Let $\a\in\NN^n$  and $V$ be a representation such that $\<\a,\ul\dim V\>=0$. Then 
$\cc_V$ is a semi-invariant on $R(\a,\a-E^t\a)$ of combined character $\chi_{\ul\dim V}$.
\end{lem}

To avoid repetition we skip the proof of this lemma since the same statement is proved later in a more general setting for $\a\in\ZZ^n$  (\ref{weightprop}).

\begin{thm}\label{c_Vspangen}
If $\a\in\NN^n$ then the space $SI(\a , \a-E^t \a )_{\chi_{\s}}$ of semi-invariants on $R(\a,\a-E^t\a)$ of combined character $\chi_{\s}$ is spanned by the semi-invariants $\cc_V$ for all modules $V$ such that $\< \a, \ul\dim V\> = 0$ and $\ul\dim V=\s$.
\end{thm}

\begin{proof} %
Let $f$ be a semi-invariant on $R(\a,\a-E^t\a)$ of weight $\s\in\NN^n$. Then
\[
	f((g^0,g^1)\z(M))=\chi_\s(g^0)\chi_\s(g^1)(f\circ\z)(M).
\]
By \ref{wrelation}, $f\circ\z$ is a semi-invariant on $R(\a)$ of weight $E\s$. Proposition \ref{orbit prop} implies that the general element of $R(\a,\a-E^t\a)$ has the form $(g^0,g^1)\z(M)$ where $M\in R(\a)$. Therefore, the above formula shows that $f$ is determined by $f\circ\z\in SI(Q,\a)$ and the weight $\s$. 
So, it suffices to find a linear combination of $C_V$'s of weight $\s$ so that the corresponding
linear combination of $c_V$'s is equal to $f\circ\z$.

By the First Fundamental Theorem (\ref{thm: FFT}), $f\circ\z$ is a linear combination of semi-invariants $c_{V_i}$ of weight $E\s$ where we may assume that each $V_i$ has support contained in the support of $\a$. Since $\chi_{E\s}=\chi_{E\ul\dim V_i}$, we have that
\begin{enumerate}
\item $(E\s)_v=(E\ul\dim V_i)_v$ for all $v$ in the support of $\a$ and
\item $(E\ul\dim V_i)_v=-\sum_{v\to w}\dim (V_i)_w\le0$ if $\a_v=0$.
\end{enumerate}
Therefore,
\[
	\g_i=E\s-E\ul\dim V_i\in\NN^n
\]
for each $V_i$. Let $I(\g_i)$ be the injective module with socle $S(\g_i)$. Then $\ul\dim I(\g_i)=E^{-1}\g_i$. So, $C_{V_i\oplus I(\g_i)}$ is a semi-invariant of $R(\a,\a-E^t\a)$ of weight
\[
	\ul\dim V_i+\ul\dim I(\g_i)=\s.
\]
Furthermore, $C_{V_i\oplus I(\g_i)}=c_{V_i\oplus I(\g_i)}=c_{V_i}$ since $c_{I(\g_i)}=1$ on $R(Q,\a)$. Therefore, $f$ is a linear combination of these determinantal semi-invariants.
\end{proof}

\begin{cor}\label{iso of rings of sis}
Let $\a\in \NN^n$.
There is an isomorphism of rings of semi-invariants
\[
	SI^{(\a,\a-E^t\a)}(Q,\a) \cong SI(Q,\a)
\]which sends $C_V$ to $c_V$ if and only if the support of $\a-E^t\a$ is contained in the support of $\a$.
\end{cor}
\begin{proof} The mapping of rings is given by the mapping
\[
	\z:R(\a)\to R(\a-E^t\a)
\]
which is equivariant with respect to the group homomorphism
\[
	(\f_0,\f_1^{-1}):G(\a)\to \Aut(P(\a))\times\Aut(P(\a-E^t\a))^{op}
\]
by Equation \ref{equivariance of zeta} in the proof of Proposition \ref{wrelation}. Therefore, $\z$ induces a homomorphism of rings of semi-invariants
\[
	\z^\ast:SI^{(\a,\a-E^t\a)}(Q,\a) \cong SI(Q,\a)
\]
By Proposition \ref{wrelation} and Remark \ref{Indeterm} this ring homomorphism is graded, sending semi-invariants of weight $\s$ to semi-invariants of weight $E\s$ and this is a 1-1 correspondence of weight cosets when the support of $\a-E^t\a$ is contained in the support of $\a$.
Therefore, it suffices to show that $\z^\ast$ induces an isomorphism
\[
	\z^\ast:SI^{(\a,\a-E^t\a)}(Q,\a)_\s \cong SI(Q,\a)_{E\s}
\]
By Remark \ref{rem:zeta sends Cv to cv}, $\z^\ast$ sends $C_V$ to $c_V$. By the First Fundamental Theorem \ref{thm: FFT}, $SI(Q,\a)$ is spanned by the functions $c_V$ for all representations $V$ with $\langle \a,\ul\dim V\rangle =0$. Theorem \ref{c_Vspangen} above tells us that $SI^{(\a,\a-E^t\a)}(Q,\a)_\s$ is spanned by the corresponding $C_V$'s. Remark \ref{Indeterm} assures us that $C_V$ has weight $\s$. Therefore $\z^\ast$ is onto.

To show that $\z^\ast$ is 1-1 take any element $f\in SI^{(\a,\a-E^t\a)}(Q,\a)_\s$ in the kernel of $\z^\ast$. Then $f$ is a semi-invariant which is trivial on $R(\a)$. But the orbit of $\z(R(\a))$ is open by Lemma \ref{lem}. Therefore, $f$ is zero on an open set. So, $f$ must be identically zero. So, $\z^\ast$ is an isomorphism as claimed.
Coversely, suppose there is a vertex $v$ in the support of $\a-E^t\a$ so that $\a_v=0$.  In that case we take $V=I(v)$ the injective envelope of the simple at $v$. Then $c_V=1$ but $C_V$ is not constant. So, the rings of semi-invariants are not isomorphic in this case.
\end{proof}


\section{Presentation spaces and their semi-invariants  for vectors $\a\in\ZZ^n$}\label{sec5}
In this section we return to study presentation spaces of arbitrary dimension vectors. First we prove existence of determinantal semi-invariants for all presentation spaces. In preparation for the virtual generic decomposition theorem, it is instructive to define and prove existence of the particular projective decomposition of $\a\in\ZZ^n$, called the \emph{canonical projective decomposition} (definition \ref{Rcan}).


\subsection{ Determinantal semi-invariants}
We now concentrate on the semi-invariants on presentation spaces which are defined using determinants and
 determine their weights.
 Only later, we will show that the rings of all semi-invariants on presentation spaces are spanned by the determinants.

The following lemma is clear for the non-negative integral vectors $\a\in \NN^n$ from \ref{fact}, however it is true for all integral vectors $\a \in \ZZ^n$.

\begin{lem} \label{MatSqu} Let $\a \in \ZZ ^n$ and let $V$ be a $Q$-representation. Then $\langle\a,\ul{\dim}V\rangle = 0$ if and only if:
$\Hom_Q (\varphi, V): \Hom_Q(P(\g_0),V) \to \Hom_Q(P(\g_1),V)$ is a square matrix for any presentation $\varphi \in R(\g_0,\g_1)$ and for any projective decomposition $E^t \a = \g_0 - \g_1$ of $\a$.
\end{lem}

\begin{proof}
$\langle\a,\ul{\dim}V\rangle = \a^t E \ul\dim V=(\g_0 -\g_1)^t \ul\dim V = (\g_0)^t \ul\dim V- (\g_1)^t \ul\dim V = \dim_{\k} \Hom_Q(P(\g_0),V) - \dim_{\k} \Hom_Q(P(\g_1),V).$
It follows that 
$\langle\a,\ul\dim V\rangle =0$ if and only if the matrix $\Hom_Q(\varphi,V)$ is square (not necessarily invertible). The dimensions of the matrix are $(\Sigma_{v\in Q_0}\dim V_v \cdot \g_{1,v}) \times (\Sigma_{v\in Q_0}\dim V_v \cdot \g_{0,v})$ since $\dim \Hom_Q(P(\g_i),V) = \Sigma_{v\in Q_0}\dim V_v \cdot \g_{i,v}$ for $i=0,1$. (For more detailed description of this matrix see the proof of Proposition \ref{weightprop}.)
\end{proof}

\begin{defn}\label{detsemi} Let  $\a\in\ZZ^n$ and $V$ a $Q$-representation such that $\langle\a,\ul{\dim}V\rangle=0$. For any projective decomposition $E^t \a = \g_0-\g_1$ of $\a$ we define, on the presentation space $R(\g_0 ,\g_1)$, the function
$$\cc^{(\g_0,\g_1)}_V:= \det (\Hom_Q(\ ,V)).$$ 

\end{defn}

\begin{prop}\label{weightprop}
Let  $\a\in\ZZ^n$ and $V$ a $Q$-representation with $\langle\a,\ul{\dim}V\rangle=0$.
\begin{enumerate}
\item  The functions $\cc^{(\g_0,\g_1)}_V$ are semi-invariants for all projective decompositions $(\g_0,\g_1)$ of $\a$.
\item The weight of $\cc_V^{(\g_0,\g_1)}$ is $(\chi_{\ul\dim V},\chi_{\ul\dim V})$.
\end{enumerate}
\end{prop}


\begin{proof} 

(1) Consider a presentation: 
$\varphi \in R(\g_0,\g_1)=\Hom_Q (P(\g_1),P(\g_0)),$
$$P(\g_1)=\coprod_{v\in Q_0} P(v)^{\g_{1,v}}\ \xrarrow{\varphi}\ \ P(\g_0)=\coprod_{v\in Q_0} P(v)^{\g_{0,v}}.$$
The map $\varphi$ is given by a matrix of size:
$(\sum_{v\in Q_0} \g_{0,v})\times (\sum_{v\in Q_0}\g_{1,v})$ with entries in $\Hom_Q(P(v),P(u))=\prod_{p:u\to v}\k$, the vector space generated by all directed paths $p:u\to v$.
For each pair of vertices $u,v$ of $Q$ and all paths $p: u\to v$,
let $\varphi_{uv,p}=\varphi_p$ denote the $\g_{0,u}\times \g_{1,v}$-matrix with coefficients in $\k$, corresponding to the $p$ coordinate of the composition:
$$P(v)^{\g_{1,v}} \xrightarrow{incl} P(\g_1)\xrightarrow{\varphi} P(\g_0) \xrightarrow{proj} P(u)^{\g_{0,u}}.$$
Then the matrix representing the map:
$$\Hom_Q(\varphi,V): \Hom_Q(P(\g_0),V) \to \Hom_Q(P(\g_1),V) $$
is a block matrix with blocks of size $(\dim V_v \cdot \g_{1,v})\times(\dim V_u \cdot \g_{0,u})$ with coefficients in $\k$: 
$$(\Hom_Q(\varphi,V))_{vu}=\sum_{p:u\to v}\Hom_{\k}(\varphi_{p},\k)\otimes_{\k}V_{p} =\sum_{p:u\to v} \varphi_{p}^{*}\otimes_{\k}V_{p},$$
where $V_{p}: V_u\to V_v$
is the map induced by the representation $V$.
\vs2
The fact that $\langle\a,\ul{\dim}V\rangle =0$ implies that the matrix $\Hom_Q(\varphi, V)$ is a square matrix by Lemma \ref{MatSqu}. Hence the determinant $\det(\Hom_Q(\varphi,V))$ is defined, and therefore $\cc^{(\g_0,\g_1)}_V$ is a polynomial function on $R(\g_0,\g_1)$.
To show that $\cc^{(\g_0,\g_1)}_V$ is a semi-invariant, we need to show that:
$$\cc^{(\g_0,\g_1)}_V((g^0,g^1)\varphi)=\chi(g^0,g^1)\cc^{(\g_0,\g_1)}_V(\varphi)$$
for some character $\chi$. Using the properties of characters from \ref{Char}, we have:

$$\cc^{(\g_0,\g_1)}_V((g^0,g^1)\varphi)=\det\Hom_Q((g^0,g^1)\varphi,V) = 
\det\Hom_Q(g^0 \varphi g^1,V) =$$

$$=(\det\Hom_Q(g^0,V))\cdot (\det\Hom_Q(\varphi,V))\cdot(\det\Hom_Q(g^1,V)).$$

Note that, in the matrix $\Hom_Q(g^i,V)$, we have that $g^i_{vv}$ is a $\g_{i,v}\times\g_{i,v}$ matrix which occurs $\dim V_v$ times, for $i=0,1$. So the above is equal to:

$$(\prod_{v \in Q_0}\det(g^0_{vv})^{\dim V_v})
\cdot (\cc^{(\g_0,\g_1)}_V(\varphi)) \cdot
(\prod_{v\in Q_0}\det(g^1_{vv})^{\dim V_v})=$$
$$=\chi^0_{\ul\dim V}(g^0)\cdot \cc^{(\g_0,\g_1)}_V(\varphi) \cdot \chi^1_{\ul\dim V}(g^1)=\chi_{(\ul\dim V,\ul\dim V)}(g^0,g^1)\cdot \cc^{(\g_0,\g_1)}_V(\varphi).$$
 \end{proof}

We now consider all projective decompositions $(\g_0,\g_1)$ of $\a$ in the directed poset $PD(\a)$, and show under which conditions the determinantal semi-invariants are nonzero.
\begin{prop}\label{homext}
Let $\a\in\ZZ^n$,  the  and let $V$ be a representation. Then the following statements are equivalent:
\begin{enumerate}
\item[i.] There exists a projective decomposition of $\a$, $(\g_0,\g_1)\in PD(\a)$ such that $\cc^{(\g_0,\g_1)}_V$ is a non-zero semi-invariant on $R(\g_0,\g_1)$.

\item[ii.] There exists a module $M$ and a projective module $P$ such that: 
(1) $\a=\ul\dim M-\ul\dim P$,
(2) $\Hom_Q(P,V)=0$,
(3) $\Hom_Q(M,V)=0$ and
(4) $\Ext_Q(M,V)=0$.
\end{enumerate}
\end{prop}
\begin{proof} (i $\then$ ii) Let $(\g_0,\g_1)$ be a projective decomposition of $\a$ such that $\cc^{(\g_0,\g_1)}_V = \det \Hom_Q(\ ,V)$ is a non-zero semi-invariant on $R(\g_0,\g_1)$ and let $\varphi$ be a general element of $R(\g_0,\g_1)$. Consider the exact sequence:
$$0\to Ker(\varphi) \to P(\g_1) \xrightarrow{\varphi} P({ \g_0}) \to Coker(\varphi) \to 0,$$
and let $M:=Coker(\varphi)$ and $P:=Ker(\varphi)$. It is easy to check that $M$ and $P$ satisfy conditions 1, 2, 3, 4.

 (ii $\then$ i)
Given $P,M$ satisfying conditions 1, 2, 3, 4, let $P_1\xrarrow{\psi} P_0\to M\to 0$ be a projective resolution of $M$, let $$\g_0 = \ul \dim(P_0/\text{rad}P_0)\text{ and }
\g_1 =\ul\dim((P_1\smallcoprod P)/\text{rad}(P_1\smallcoprod P)).$$
Consider the presentation space:
\[
	R(\g_0,\g_1)=\Hom_Q(P\smallcoprod P_1,P_0)=\Hom_Q(P,P_0)\times \Hom_Q(P_1,P_0),
\]
and let $\varphi: =(0,\psi)\in R(\g_0,\g_1)$. Then the mapping
\[
	\Hom_Q(\varphi,V)=\Hom_Q((0,\psi),V):\Hom_Q(P_0,V)\to \Hom_Q(P\smallcoprod P_1,V)
\]
is a monomorphism by (3) and an epimorphism by (2) and (4). Consequently, $\det\Hom_Q(\varphi,V)\neq0$, i.e. $\cc^{(\g_0,\g_1)}_V\neq0$. 
\end{proof}

\subsection{ Stability in presentation spaces}

This subsection is devoted to investigating the general elements in the presentation spaces.
We prove the Stability Theorem, which asserts that the general element in the presentation space is homotopically equivalent to
an element in the space corresponding to a minimal decomposition of $\a$. 

We recall that the direct sum of homomorphisms gives a mapping
\[
	\smallcoprod:R(\g_0,\g_1)\smallcoprod R(\g_0',\g_1')\to R(\g_0+\g_0',\g_1+\g_1').
\]

\begin{defn}\label{StabMaps}
We define the \emph{stabilization maps} for any $\g_0,\g_1,\g\in\NN^n$
$$St_{\g}^{(\g_0,\g_1)}:R(\g_0,\g_1) \to R(\g_0+\g,\g_1+\g)$$
as $\,St_{\g}^{(\g_0,\g_1)}(\varphi):=\varphi\coprod 1_{P(\g)}$ for each $\varphi \in R(\g_0,\g_1)=\Hom_Q(P(\g_1),P(\g_0))$.
\end{defn}

\begin{thm}[Stability theorem]\label{thm:stability theorem}
Given any projective decomposition $E^t\a=\g_0-\g_1$ of $\a\in\ZZ^n$, the general element of $R(\g_0,\g_1)$ is isomorphic to an element in the image of the stabilization map $R^{min}(\a)\to R(\g_0,\g_1)$.
\end{thm}

\begin{proof}
Let  $\a\in\ZZ^n$ and $E^t\a=\g_0-\g_1$ be a projective decomposition of $\a$. Suppose it is not minimal. Then $(\g_0,\g_1) = (\g^{min}_0+\g,\g^{min}_1+\g)$ where $(\g^{min}_0,\g^{min}_1)$ is the minimal projective decomposition of $\a$.

Let $\varphi \in R(\g_0,\g_1)=R(\g^{min}_0+\g,\g^{min}_1+\g)$ be a general element.
We will show that 
$$\varphi = (g^0,g^1)( \varphi^{min}\smallcoprod 1_{P(\g)})=(g^0,g^1)(St_{\g}^{(\g_0^{min}, \g_1^{min})}(\varphi^{min})) = (g^0,g^1) St_{\g}^{min}(\varphi^{min})$$ 
where $\varphi^{min}$ is an element in $R(\g_0^{min}, \g_1^{min})=R^{min}(\a)$.

By the above projective decomposition of $\a$, we have:
$$\Hom_Q(P(\g_1)\smallcoprod P(\g))\xrarrow{\varphi} \hom_Q(P(\g_0)\smallcoprod P(\g)).$$

So $\varphi$ can be viewed as a matrix

\[\varphi= \begin{pmatrix}
	f & h \\
	g & r 
\end{pmatrix}:
{P(\g_{1})} \smallcoprod P(\g)\to
 {P(\g_{0})} \smallcoprod P(\g),
\]
where $r: P(\g) \to P(\g)$ is an isomorphism since $\varphi$ is a general element (\ref{RankGenPr}). 
\[\text{From }\ \ \ \ 
\begin{pmatrix}
	f & h \\
	g & r 
\end{pmatrix}
=
\begin{pmatrix}
	1_{P(\g_0)} & hr^{-1} \\
	0 & 1_{P(\g)}
\end{pmatrix}
\begin{pmatrix}
	f-hr^{-1}g & 0 \\
	0 & 1_{P(\g)}  
\end{pmatrix}
\begin{pmatrix}
	1_{P(\g_1)}  & 0 \\
	g & r 
\end{pmatrix}
\]
it follows that 
\[\varphi ={(g^0,g^1)}
\begin{pmatrix}
	f-hr^{-1}g & 0 \\
	0 & 1_{P(\g)}  
\end{pmatrix}=
{(g^0,g^1)}
\begin{pmatrix}
	\varphi^{min} & 0 \\
	0 & 1_{P(\g)}  
\end{pmatrix} = {(g^0,g^1)}{St^{min}_{\g}}(\varphi^{min}),
\]
where
\[{g^0} =\begin{pmatrix}
	1_{P(\g_0)} & hr^{-1} \\
	0 & 1_{P(\g)}
\end{pmatrix} \in {\Aut (P(\g_0)\smallcoprod P(\g))},
\]
\[ 
{g^1} =\begin{pmatrix}
	1_{P(\g_1)}  & 0 \\
	g & r 
\end{pmatrix}\in {\Aut (P(\g_1)\smallcoprod P(\g))^{op}},
\]
and $\varphi^{min} = f-hr^{-1}g \in R(\g_0^{min},\g_1^{min}) = R^{min}(\a)$.
\end{proof}

 \begin{rem}
For $\a\in\NN^n$ this says that, for the minimal projective resolution $0\to P_1\to P_0\to M \to 0$ of a general module $M$ of dimension $\a$, $P_0$ and $P_1$ have no summands in common. (Apply the above theorem to $R(\a,\a-E^t\a)$ and use Remark \ref{rem:general properties of reps and maps} to pass from general properties of elements of $R(\a,\a-E^t\a)$ to general properties of modules.)
\end{rem}

\subsection{ Canonical presentation spaces for $\a\in\ZZ^n$}\label{subsec canonical presentation space}

It was  observed in \ref{PresSp} that for $\a\in\NN^n$, the canonical presentation is an element of the presentation space $R(\a,\a-E^t\a)$ and there is a close relationship between the classical representation space $R(\a)$ and the presentation space $R(\a,\a-E^t\a)$ (section \ref{sec4}). For $\a\in\ZZ^n$ we generalize this special presentation space to the canonical presentation space $R^{can}(\a)$ which in the case of $\a\in\NN^n$ turns out to be the same as $R(\a,\a-E^t\a)$.

\begin{lem}

Let $\a\in\ZZ^n$ be fixed and let $E^t\a=\g_0-\g_1$ be the minimal decomposition. Let ${\f}:P(\g_1)\to P(\g_0)$ be a general element of $R^{min}(\a)=\Hom_Q(P(\g_1),P(\g_0))$ with kernel and cokernel $P,M$:
\[
 	0\to P\to P(\g_1)\xrarrow{\f}P(\g_0)\to M\to 0.
\]
Then: (a)  $P$ must be a direct summand of $P(\g_1)$, and (b)  $\Hom_Q(P,M)=0.$ 
\end{lem}
\begin{proof} (a)
Note that $P$ must be a direct summand of $P(\g_1)$ since $Im \,\phi \subset P(\g_0)$ is projective. Let $P(\g_1)=P'\coprod P$.\\
(b) To see $\Hom_Q(P,M) = 0,$ 
let  $f:P\to M$ be any nonzero homomorphism. Then $f$ lifts to a homomorphism $\psi:P\to P(\g_0)$ whose image is not contained in the image of $\phi$. This implies that the homomorphism $\f+\psi:P'\coprod P\to P(\g_0)$ has image strictly containing the image of $\f$ and therefore has rank greater than the rank of $\f$. This gives a contradiction, since $\f+\psi$ is a specialization of the general map $\f$.
\end{proof}
 
Let $\g=\ul\dim (P/radP)$ so that $P\cong P(\g)$.
Let $\mu=\ul\dim M$.

\begin{lem}\label{unique}
Let $\a \in \ZZ^n$. Then the vectors $\mu=\ul\dim M$ and $\g=\ul\dim (P/rad P)$ satisfy the following properties. These properties determine $\mu,\g\in\NN^n$ uniquely.
\begin{enumerate}
\item $\mu,\g$ have disjoint support.
\item $\a=\mu-\Eti\g$
\end{enumerate}
Furthermore, in the special case when $\a\in\NN^n$, we have $\mu=\a$ and $\g=0$.
\end{lem}

\begin{proof}
(1) follows from the fact that $\Hom_Q(P,M)=0$ and (2) follows from a dimension counting argument. We are left to prove the uniqueness of $\mu,\g$.

Let $\a\in\NN^n$ and suppose that we are given a decomposition $\a=\mu-\Eti\g$. Let $v$ be a vertex in the support of $\g$ which is minimal with respect to the partial ordering of the vertices of $Q$.  Then $\a_v<0$,  using Useful fact~\ref{fact}.4, for instance.
Therefore, $\g=0$ proving uniqueness.

Now proceed by induction on the number of negative coordinates of $\a$.  If $\a$ has negative coordinates then let $v$ be minimal so that $\a_v<0$. Then $\a'=\a+|\a_v|(E^t)^{-1}e_v$ has fewer negative coordinates than $\a$ so we have a unique decomposition $\a'=\mu-(E^t)^{-1}\g$. Then we must have $\mu_v=\g_v=0$ and $\a=\mu-(E^t)^{-1}(\g+|\a_v|e_v)$ is the unique admissible decomposition of $\a$.
\end{proof}

This lemma motivates the following

\begin{defn}\label{Rcan}
Let $\a\in \ZZ^n$.  \emph{The canonical projective decomposition of $\a$} is defined as $(\mu,\mu-E^t\mu +\gamma)$, where $\mu, \g \in \NN^n$ are uniquely defined vectors as in Lemma~\ref{unique}.  Note that $\mu - E^t \mu \in \NN^n$ by Useful fact~\ref{fact}.6.  
We also define the {\em canonical \repspace} 
$$
R^{can}(\a ):=R(\mu,\mu-E^t\mu+\g).
$$
\end{defn}

\begin{rem}
For $\a\in \NN^n$, we have $R^{can}(\a)=R(\a,\a-E^t\a)$, which is the special case we considered in the previous section.
\end{rem}

\begin{eg}

We now illustrate  \ref{unique} and \ref{Rcan} on our Example \ref{Example}. 
If 
$$\a= (1, 2, -3)^t,\text{ then } \g = (0,0,3)^t, \mu=(1,2,0)^t$$
$$R^{can}(\a)=R^{can}((1,2,-3)^t)=R((1,2,0)^t,(0,1,7)^t),$$
$$R^{min}(\a)=R^{min}((1,2,-3)^t)=R((1,1,0)^t,(0,0,7)^t).$$
\commentout{
\[\text{If  }\a =
\left[
\begin{array}{ccc}
  1     \\
  1     \\
 -2      
\end{array}
\right], \text{ then }
\mu =
\left[
\begin{array}{ccc}
  1     \\
  1     \\
0      
\end{array}
\right] \text{ and }
\g =
\left[
\begin{array}{ccc}
 0     \\
  0     \\
2     
\end{array}
\right], 
\]

\[
R^{can}(\a) =R^{can}\left(
\left[
\begin{array}{ccc}
  1     \\
  1     \\
 -2      
\end{array}
\right]\right) =R\left(
\left[
\begin{array}{ccc}
  1     \\
  1     \\
 0      
\end{array}
\right],
\left[
\begin{array}{ccc}
 0     \\
  1     \\
2     
\end{array}
\right]\right), 
\]
\[R^{min}(\a)=
R^{min}\left(
\left[
\begin{array}{c}
  1     \\
  1     \\
 -2      
\end{array}
\right]\right)=R\left(
\left[
\begin{array}{c}
  1     \\
  0     \\
0     
\end{array}
\right],
\left[
\begin{array}{c}
  0     \\
  0     \\
 2      
\end{array}
\right]\right).
\]
}
\end{eg}

\begin{prop}\label{can decomposition prop}
The general element of $R^{can}(\a)$ is isomorphic to the direct sum of the canonical presentation $p_M$ of the general element $M$ of $R(\mu)$ and the unique element of $R(0,\g)$.
\end{prop}

\begin{proof}
By the Stability Theorem \ref{thm:stability theorem}, the general element of $R^{can}(\a)$ is isomorphic to a stabilized element of 
$R^{min}(\a)$. Therefore the general element 
$$P(\mu-E^t\mu +\g) \to P(\mu)$$
will have kernel $P(\g)$ which is a direct summand. Consequently, the general element of $R^{can}(\a)$
 is a direct sum of the unique element of $R(0,\g)$ and an element of $R^{can}(\mu)$. Since $\mu\in \NN^n$, Proposition \ref{orbit prop} now applies. Therefore the general element of $R^{can}(\mu)$ lies in the orbit of $\z(R(\mu))$, i.e. it is isomorphic to a canonical presentation of an element of $R(\mu)$.
\end{proof}

\section{Virtual representation spaces and virtual semi-invariants}\label{sec6}
In this section we again deal with integral vectors $\a\in\ZZ^n$, defining  the virtual representation space as the direct limit of presentation spaces.  Similarly, we define the rings of virtual semi-invariants on the virtual representation spaces as the inverse limits of the rings of semi-invariants on the presentation spaces. Finally, we prove the Virtual Generic Decomposition Theorem (generalizing Proposition \ref{can decomposition prop}) and the Virtual First Fundamental Theorem.

\subsection{ Virtual representation space}\label{VirRep}
We recall that the {stabilization maps}:
$$
\Hom_Q(P(\g_1),P(\g_0))\to\Hom_Q(P(\g_1)\smallcoprod P(\g),P(\g_0)\smallcoprod P(\g)),
$$
are the maps which send  $\varphi$ to  $\varphi\coprod 1_{P(\g)}.$

Let $\a\in\ZZ^n$. Then the set of all representation spaces $\{R(\g_0 ,\g_1 )\}_{(\g_0 ,\g_1)\in PD(\a )}$ 
and the stabilization maps form a directed system.
We define  {\em the virtual representation space} as the direct limit over $PD(\a )$:
$$
R^{vir}(\a) = \lim_\rarrow R(\g_0 ,\g_1).
$$
(Notice that  a given pair $(\g_0 ,\g_1 )$ of dimension vectors belongs to exactly one partially ordered set $PD(\a )$, namely the one where $\a=(E^t)^{-1}(\g_0-\g_1)$.)

\subsection{ Virtual semi-invariants}\label{virsemi}

The rings $SI^{(\g_0 ,\g_1 )}(Q,\a )$ and the restriction maps induced by stabilizing define an inverse system of rings on the directed partially ordered set $PD(\a )$.
We define the ring of \emph{virtual semi-invariants} as the inverse limit over $PD(\a)$:
$$SI^{vir} (Q,\a ):=\lim_\larrow SI^{(\g_0 ,\g_1 )}(Q,\a ).$$

\noindent In other words, a {\em virtual semi-invariant} on $R^{vir}(\a)$ is a function $f^{vir}$ induced by a family of stabilization compatible semi-invariants $f^{(\g_0,\g_1)}$ on the representation spaces $R(\g_0 ,\g_1)$.

The definition of invariants $\cc_V$ induced by the determinants given in Proposition \ref{weightprop}
generalizes to virtual semi-invariants. More precisely, we have

\begin{prop}\label{virweightprop}
Let  $\a\in\ZZ^n$ and $V$ a $Q$-representation such that $\langle\a,\ul{\dim}V\rangle=0$. Then:
\begin{enumerate}
\item The family of semi-invariants $\{\cc_V^{(\g_0,\g_1)}\}_{(\g_0,\g_1)\in PD(\a)}$ is compatible with stabilizations, and thus it defines an element $\cc_V^{vir}\in SI^{vir}(Q, \a)$.
\item  The induced semi-invariant $\cc^{vir}_V$ on the virtual representation space $R^{vir}(\a)$ has combined character $\chi_{\ul\dim V}$.
\end{enumerate}
\end{prop}

We proceed to analyze the general elements in the virtual representation spaces.

\begin{cor}\label{cor516}
Let $\a\in\ZZ^n$ and let $V$ be a representation of $Q$. Then the following are equivalent.
\begin{enumerate}
\item $\cc_V\neq0$ on $R^{min}(\a)$.
\item $\cc_V^{(\g_0,\g_1)}\neq0$ on $R(\g_0,\g_1)$ for all projective decompositions $(\g_0,\g_1)\in PD(\a)$.
\item $\cc_V^{(\g_0,\g_1)}\neq0$ on $R(\g_0,\g_1)$ for some projective decomposition $(\g_0,\g_1)\in PD(\a)$.
\item $\cc_V^{vir}\neq0$ on $R^{vir}(\a)$.
\end{enumerate}
\end{cor}

\begin{rem}
This is an extension of Proposition \ref{homext}.
\end{rem}

\begin{proof}
If $\cc_V\neq0$ on $R^{min}(\a)$ then the composition $R^{min}(\a)\to R(\g_0,\g_1)\xrarrow{\cc_V}\k$ is nonzero. So $(1)\then(2)$.

Clearly, $(2)\then(3)$ and $(3)\iff (4)$ by definition of the virtual semi-invariant $\cc_V^{vir}$. Finally, $(3)\then(1)$ by the stability theorem. If $\cc_V^{(\g_0,\g_1)}\neq0$ on $R(\g_0,\g_1)$ then $\cc_V^{(\g_0,\g_1)}\neq0$ on the general element of $R(\g_0,\g_1)$ which is equivalent to an element of $R^{min}(\a)$ by stability.
\end{proof}

\begin{defn}
We define the \emph{$\ZZ$-support} of $\cc_V$ to be the set of all $\a\in\ZZ^n$ so that any of the equivalent conditions of Corollary \ref{cor516} hold:
\[
	supp_\ZZ(\cc_V):=\{\a\in\ZZ^n\st \cc_V^{vir}\neq0 \text{ on } R^{vir}(\a)\}
\]
\end{defn}

\begin{lem}\label{suppcor} If $\b=\ul\dim V$ is sincere, i.e., $\b_v\neq0$ for all $v\in Q_0$, then $$supp_\ZZ(\cc_V )=\{ \a\in\NN^n\ |\ \<\a,\b\>=0\ {\rm and}\ \exists M\in R(\a )\ {\rm such\ that} \Hom_Q(M,V)=0.\}$$\end{lem}
\begin{proof}  This follows from Proposition~\ref{homext}, Corollary \ref{cor516} and Useful fact~\ref{fact}.1.\end{proof}

\subsection{ Virtual Generic Decomposition Theorem}\label{VirGenDec}

Let us make preparations to state the Virtual Generic Decomposition Theorem.
We want to extend the notions of $hom_Q$ and $ext_Q$ to include shifted projective module such as $P(\g)[1]$. This is the projective complex $P(\g)\to 0$ which is the unique element of $R(0,\g)$. We note that shifted projectives are uniquely determined up to isomorphism by their dimension vector which is negative:
\[
	\ul\dim P(\g)[1]=-(E^t)^{-1}\g.
\]
We use the notation $hom_{D^b}(\a,\b[1])=ext_Q(\a,\b)$ and $ext_{D^b}(\a[1],\b)=hom_Q(\a,\b)$ and, in general,
\[
	ext_{D^b}(\a[p],\b[q])=hom_{D^b}(\a[p],\b[q+1]):=\begin{cases}
		ext_Q(\a,\b) &\text{if } p=q\\
		hom_Q(\a,\b) & \text{if } p=q+1\\
		0 &\text{otherwise}
	\end{cases}
\]
for all $\a,\b\in\NN^n$. In particular, $ext_{D^b}(\pi(\g)[1],\b)=0$ for $\b\in\NN^n$ if and only if $\b,\g$ have disjoint supports.

Let  $\a\in\ZZ^n$. Consider the canonical representation space as defined in \ref{Rcan}.
$$
R^{can}(\a ):=R(\mu,\mu-E^t\mu+\g ).
$$
The dimension vector $\mu=\ul\dim M$ has a generic decomposition
\[
	\mu=\sum \b_i
\]
where $\b_i$ are Schur roots with the property that $ext_Q (\b_i,\b_j)=0$ for all $i\neq j$. We recall that \emph{Schur roots} are dimension vectors $\b_i$ so that the general representation of dimension $\b_i$ is indecomposable, \ref{Schur}. Thus $M$ decomposes as $M\cong \coprod M_i$ with $\ul\dim M_i=\b_i$ where $M_i$ are indecomposable modules which do not extend each other.


\begin{thm}[Virtual Generic Decomposition]\label{thm:generic decomposition}
Any $\a\in\ZZ^n$ has a unique decomposition of the form
\[
	\a=\b_1+\b_2+\cdots+\b_k-(E^t)^{-1}\g
\]
where
\begin{enumerate}
\item $\b_1,\cdots,\b_k,\g\in\NN^n$,
\item $\b_i,\g$ have disjoint support for all $i$,
\item $ext_{D^b}(\b_i,\b_j)=0$ for all $i\noteq j$ and
\item each $\b_j$ is a Schur root.
\end{enumerate}
Furthermore, the general element $f:P_1\to P_0$ of the canonical \repspace\ $R^{can}(\a)$ is homotopy equivalent to a direct sum of projective complexes which are either
\begin{enumerate}
\item[(i)] minimal resolutions of indecomposable modules $M_j$ with $\ul\dim M_j=\b_j$ or 
\item[(ii)] complexes of the form $P(v_i)[1]=(P(v_i)\to 0)$ with $\ul\dim P(v_i)[1]=-\Eti e_{v_i}$.\qed
\end{enumerate}
\end{thm}

\begin{proof}
Proposition \ref{can decomposition prop} tells us that the general element $f:P_1\to P_0$ has $\ker f=P(\g)$. The classical generic decomposition theorem (\ref{classical generic decomp}) gives us the stated decomposition of $M=\coker f$.
\end{proof}

\subsection{ Virtual First Fundamental Theorem}\label{subsec VFFT}

\begin{thm}[Virtual First Fundamental Theorem]\label{thm: vFFT}
For any $\a\in\ZZ^n$ the ring of virtual semi-invariants on $R^{vir}(\a)$ is generated by the semi-invariants $\cc_V^{vir}$ for all modules $V$ so that $\<\a,\ul\dim V\>=0$ and $\<\b_j,\ul\dim V\>=0$ for all $\b_j$ in the generic decomposition of $\a$. Consequently, we have a graded decomposition of the ring of virtual semi-invariants
\[
	SI^{vir}(Q,\a)=\bigoplus_\s SI^{vir}(Q,\a)_{\chi_\s}
\]
where the sum is over all $\s\in\NN^n$ with support disjoint from the support of $\g$ such that $\<\a,\s\>=\<\b_i,\s\>=0$ for all $\b_i$ in the generic decomposition of $\a$.
\end{thm}

\begin{proof}
By Proposition~\ref{can decomposition prop} the general element of $R^{can}(\alpha )$ is a direct sum of a presentation of a module $M$ with dimension vector $\mu$ and the shifted projective module $P(\g)[1]$. This shows that every semi-invariant on $R^{can}(\a)$ restricts to a semi-invariant on $R(\mu , \mu -E^t\mu )$ which determines it uniquely. But $SI(\mu , \mu -E^t\mu )$ is spanned by semi-invariants $\cc_V$ by Theorem~\ref{c_Vspangen}. And $C_V$ extends to $R^{can}(\a)$ if and only if $\<\a,\ul\dim V\>=0$. Furthermore, $C_V$ will be trivial on the general element of $R^{can}(\a)$ unless $\<\b_i,\ul\dim V\>=0$ for each $\b_i$ in the generic decomposition of $\a$. 

For a general pair $(\g_0 ,\g_1 )\in PD(\a )$ the elements in some Zariski open set in $R(\g_0 ,\g_1 )$ are the direct sum of an identity map on some projective module and a map from $R^{min}(\a )$. Since a semi-invariant is determined by its restriction to the open set, the result follows.
\end{proof}

\subsection{ Virtual Saturation Theorem}\label{sec7}
The original Saturation Theorem ~\ref{satthm} gives all non-negative integral vectors $\a$, such that the classical representation space $R(\a)$ has a non-zero semi-invariant of a prescribed weight. In this paper, we describe all integral vectors $\a$, such that the virtual representation space $R^{vir}(\a)$ has a virtual semi-invariant of a prescribed weight.

Following the classical definition of the support \ref{Supp} for the weights of semi-invariants, we give the following definition
\begin{defn}The ${\ZZ}$-{\em support} of a vector $\b\in \NN^n$ is defined to be:
$$supp_{\ZZ}(\b)= \{\a\in \ZZ^n|\  SI^{vir}(Q,\a)_{\chi_{\b}}\neq 0\}.$$
\end{defn}
As a corollary to the virtual first fundamental theorem, we have the following description of the supports of semi-invariants.
\begin{cor} The ${\ZZ}$-{\em support} of a vector $\b\in \NN^n$ is defined to be:
$$supp_{\ZZ}(\b)= \{\a\in \ZZ^n|\ \cc^{vir}_V\neq 0 \text{ on } R^{vir}(\a) \text{ for some module } V \text{ with } \ul\dim V = \b\}.$$
\end{cor}

\begin{rem}\label{OpenCondCv}
By Corollary \ref{cor516}, $\cc^{vir}_V \neq 0$ on $R^{vir}(\a)$ if and only if
$\cc_V\noteq0$ on $R(\g_0,\g_1)$ for some fixed $(\g_0,\g_1)\in PD(\a)$. This is an open condition on $V$. Therefore, if it holds for some choice of $V$ then it will hold for a general representation of dimension $\b = \ul\dim V$.
\end{rem}

The Saturation and Generalized Saturation theorems describe supports of semi-invariants and virtual semi-invariants as $\NN^n\cap D(\b)$ and $\ZZ^n\cap D(\b)$ respectively. 
The sets $D(\b)\subset \RR^n$ were already defined in \ref{DefDb}, but we give now a more detailed description together with a description of a particular $D(\b)$ for the example \ref{Example}.

Let $\b\in \NN^n$. Define $H(\b)$ to be the hyperplane in the root space $\RR^n$ given by
\[
	H(\b):=\{\a\in\RR^n\st \<\a,\b\>=0\}
\]
Let $H_+(\b),H_-(\b)\subseteq\RR^n$ be the closed half-spaces given by
\[
	H_+(\b):=\{\a\in\RR^n\st \<\a,\b\>\geq0\},\ 
	H_-(\b):=\{\a\in\RR^n\st \<\a,\b\>\leq0\}.\text{ Then}
\]
$$
D(\b )=H(\b )\cap \bigcap_{\b'\into\b} H_-(\b' ).
$$

\noindent Here we recall that $\b'\into\b$ means that the general representation of dimension $\b$ has a subrepresentation of dimension $\b'$. 
It also has a quotient of dimension $\b''=\b-\b'$ and we write $\b\onto\b''$. Since $\<\a,\b\>=\<\a,\b'\>+\<\a,\b''\>$ we see that
\[
	H(\b)\cap H_-(\b')=H(\b)\cap H_+(\b'').
\]

\begin{eg} Again, we illustrate using Example \ref{Example}.
Let $\b=(0,1,2)^t.$ Then $$D((0,1,2)^t)=
H((0,1,2)^t)\cap H_-((0,0,1)^t)\cap H_-((0,0,2)^t)=$$
$$H((0,1,2)^t)\cap H_+((0,1,0)^t)\cap H_+((0,1,1)^t).$$
\commentout{
\[
\text{Let }\b =
\left[
\begin{array}{ccc}
  1      \\
  1      \\
  1      
\end{array}
\right]. \text{ Then }
D\left(\left[
\begin{array}{ccc}
  1   \\
  1   \\
  1   
\end{array}
\right]\right)=
\]
\[
H\left(\left[
\begin{array}{ccc}
  1   \\
  1   \\
  1   
\end{array}
\right]\right)\cap
H_-\left(\left[
\begin{array}{ccc}
  0   \\
  1   \\
  1   
\end{array}
\right]\right)\cap
H_-\left(\left[
\begin{array}{ccc}
  0   \\
  0   \\
  1   
\end{array}
\right]\right)=
\]
\[
H\left(\left[
\begin{array}{ccc}
  1   \\
  1   \\
  1   
\end{array}
\right]\right)\cap
H_+\left(\left[
\begin{array}{ccc}
  1   \\
  0   \\
  0   
\end{array}
\right]\right)\cap
H_+\left(\left[
\begin{array}{ccc}
  1   \\
  1   \\
  0   
\end{array}
\right]\right).
\]
}
Therefore:
$$D(\b) =\{\ \a\in \RR^3\ |\  2\a_3=3\a_2+\a_1,\ \a_2 \ge\a_1\ \}$$

\end{eg}

The following proposition follows immediately from the definition.

\begin{prop}\label{convex}
The set $D(\b)$ is a closed and convex subset of the hyperplane $H(\b)$ for any nonzero $\b\in\NN^n$.\qed
\end{prop}

The Saturation Theorem for $\a\in\ZZ^n$ follows from the original Saturation Theorem of Derksen and Weyman  (Theorem~\ref{satthm} in this paper) and the following lemmas.

\begin{lem}\label{tfae}
 Let $P(v)$ be an indecomposable projective and $\b\in\NN^n$. Then the following conditions are equivalent.
\begin{enumerate}
\item $\b_v=0$
\item $\ul\dim P(v)\in D(\b)$
\item $-\ul\dim P(v)\in D(\b)$.
\end{enumerate}
\end{lem}

\begin{proof}  If $\b_v=0$ then $\b'_v=0$ for all $\b'\into\b$.
The rest of the proof follows from the fact that for any indecomposable projective $P(v)$:

$\langle \ul\dim P(v),\b\rangle = (\ul\dim P(v))^t E \b = (E^t \ul\dim P(v))^t \b = (\ul\dim S(v))^t \b = \b_v. $
\end{proof}

\begin{lem}\label{negvert}
 If $\a\in D(\b)$ and $\a_v<0$ then there is a vertex $w$ in the support of the injective envelope of $S(v)$ so that $\b_w=0$.
\end{lem}

\begin{rem}
We have $w\in supp(I(v))$ if and only if there is a path from $w$ to $v$. Since $Q$ has no oriented cycles this implies $w\leq v$.
\end{rem}

\begin{proof}
Suppose not. Then there is a vertex $v$ of $Q$ where $\a_v<0$ but $\b_w>0$ for all $w\leq v$ in the support of the injective envelope of $S(v)$, that is, for all $w$ having a path to $v$. Let $v$ be minimal with this property.
We have $Hom_Q (V, I(v))=V_v\ne 0$, so there are non-zero homomorphisms from $V$ to $I(v)$. Let us choose such a homomorphism with image of maximal length, and
let $L$ be its image. Then $\b\onto\g$ where $\g=\ul\dim L$ so $D(\b)\subseteq H_+(\g)$. In other words, $\<\a,\g\>\geq0$. But an injective resolution of $L$ is given by
\[
	0\to L\to I(v)\to \coprod I(w_i)
\]
where $w_i<v$. By minimality of $v$ we have $\a_{w_i}\geq0$. So,
\[
	\<\a,\g\>=\a_v-\sum \a_{w_i}<0
\]
which is a contradiction.
\end{proof}

\begin{eg} In the Example \ref{Example} we can see that
$\a = (-1,0,-2)^t \in D((0,1,2)^t)$, but $\a_1, \a_3 < 0$.
This is possible since there is a path from $w=1$ to $v=3$ (and to $v=1$). Also, $-\ul\dim P(v_1)=(-1,-1,-2)^t\in D((0,1,2)^t)$ since $\b_1=0$.

\end{eg}

In the following lemma we compare the supports for semi-invariants on the classical representation space as defined in \ref{Supp}, with the supports of semi-invariants on presentation spaces, which will be used in the proof of the Virtual Saturation Theorem.

\begin{lem} \label{NSupEb} Let $\b\in\NN^n$. Then $supp_{\NN}(E\b) =\NN^n \cap supp_{\ZZ}(\b)$.\label{NNZZSupp}
\end{lem}

\begin{proof}
$(\subseteq)$ Let $\a \in supp_{\NN}(E\b)$. Then $\a \in \NN^n$ and there exists a non-zero semi-invariant $p_V$ on $R(\a)$ of weight $\chi_{E\b}$. Furthermore, (by \ref{thm: FFT})
$$c_V= \det\Hom_Q(p_{-}, V): R(\a) \to \k,$$
for some module $V$ with $\ul\dim V =\b$, and the canonical projective presentation for $M \in R(\a)$:
$$ P_1 \xrightarrow{p_M} P_0 \to M \to 0.$$
Let $\g_0 =\ul\dim(P_0/radP_0)$ and $\g_1 =\ul\dim(P_1/radP_1)$. Then
$p_M \in \Hom_Q(P(\g_1),P(\g_0))$ $ = R(\g_0,\g_1,)$ and $\cc_V = \cc^{(\g_0,\g_1)}_V$ can be viewed as a non-zero semi-invariant for the projective decomposition $(\g_0,\g_1)$ of $\a$. Hence $\a \in \NN^n \cap supp_{\ZZ}(\b)$.

$(\supseteq)$ Conversely, let $\a \in \NN^n \cap supp_{\ZZ}(\b)$ and let $\cc^{(\g_0,\g_1)}_V$ be a non-zero semi-invariant. Also, since $\a \in \NN^n$, there is a projective presentation $0\to P_1\to P_0 \to M \to 0$, where $\ul\dim M = \a$, $P_1=P(\g'_1), P_0 = P(\g'_0)$, where $(\g'_0,\g'_1)$ is a projective decomposition of $\a$. By stabilization we may assume that $\g_0=\g_0'$ and $\g_1=\g_1'$. Since $\cc^{(\g_0,\g_1)}_V$
 is nonzero, there exists a map $f:P_1\to P_0$ so that $\Hom_Q(f,V)$ is an isomorphism. Also 
 $p_M:P_1\to P_0$ is a monomorphism. Since both of these conditions are Zariski open there exists a monomorphism $f:P_1\to P_0$ so that 
 $\Hom_Q(f,V)$ is an isomorphism. This implies that $\Hom_Q(M,V)=\Ext_Q(M,V)=0$
  for $M=P_0/fP_1$. So, $\a=\ul\dim M\in supp_\NN(E\b)$.

\end{proof}

\begin{thm}[Virtual Saturation Theorem]\label{virsatthm}
Let $\b \in \NN^n$.  Then
$$
supp_{\ZZ}(\b)=\ZZ^n\cap D(\b).
$$
\end{thm}

\begin{proof}  We proceed in steps.

\ul {Claim 1.}
$\NN^n \cap supp_{\ZZ}(\b) = \NN^n \cap D(\b).$

\noindent This follows by the Lemma \ref{NNZZSupp} and by the Derksen-Weyman Saturation Theorem~\ref{satthm}.

\ul{Claim 2.} Suppose $\b$ is sincere. Then $supp_{\ZZ}(\b)=\ZZ^n\cap D(\b).$ 

\noindent In that case Lemma~\ref{suppcor} states that  $supp_\ZZ(\cc_V)\subset \NN^n$ whenever $\ul\dim V = \b$.
 So $$supp_{\ZZ}(\b) = \bigcup_{V\in R(\b)}supp_\ZZ(\cc_V) = \NN^n \cap supp_{\ZZ}(\b).$$
By Lemma \ref{NSupEb} and the classical saturation theorem \ref{satthm}, $$\NN^n\cap D(\b) = supp_{\NN}(E\b)= \NN^n\cap D(\b).$$
But $D(\b)\subset\NN^n$ by Lemma~\ref{negvert}. So 
$\NN^n\cap D(\b)= \ZZ^n\cap D(\b)$ proving the claim.

For the remainder of the proof, we use the fact that $D(\b)$ is closed under addition.
If $\b$ is not sincere, let $P$ be the sum of projective covers (comp. \cite{[ASS]}, section I.5) of all vertices not in the support of $\b$ and let $\g=\ul\dim P$. Then $\g^tE\b=0$. For any $\a$ in the support of $\cc_V$, Proposition~\ref{homext} implies there is an $m\geq0$ such that $\a+m\g\in\NN^n$ and also lies in $supp_\ZZ(\cc_V)$.  Hence  $\a+m\g\in D(\b)$. But Lemma \ref{tfae} says that $-\g\in D(\b)$. So $\a=(\a+m\g)+m(-\g)\in D(\b)$.

Conversely, suppose that $\a\in D(\b)$ with $\b$ not sincere. Then Lemmas \ref{tfae} and \ref{negvert} imply that $\a+m\g\in\NN^n\cap D(\b) = supp_{\NN}(E\b)$ for some integer $m$. So there are modules $M\in R(\a+m\g)$ and $V\in R(\b)$ so that
$\Hom_Q(M,V)=0=\Ext^1_Q(M,V).$
Then $P^m, M$ satisfy the conditions of Proposition~\ref{homext} making $\a= (\a+m\g)-m\g$ an element of the support of $\cc_V$.
\end{proof}
\begin{cor}\label{cor to VirSat}
For any nonzero $\b\in\NN^n$,
$supp_\ZZ(\b)=supp_\ZZ(C_V)$
for $V$ in a nonempty open subset of $R(\b)$.
\end{cor}

\begin{proof}
By the virtual saturation theorem, $supp_\ZZ(\b)=\ZZ\cap D(\b)$. By definition, $D(\b)=H(\b)\cap \bigcap_{\b'\into\b} H_-(\b')$.  This  closed cone is the covex hull of a finite number of rays. These rays lie on intersections of transverse hyperplanes defined over $\QQ$. So they contain elements of $\QQ^n$ and therefore elements of $\ZZ^n$. By Remark \ref{OpenCondCv}, for each of these integer vectors $\a_i$, there is an open subset $U_i\subset R(\b)$ so that $\a_i\in supp_\ZZ(V)$ for all $V\in U_i$. Then, for any $V\in\cap U_i$,  $supp_\ZZ(V)$ contains all of the $\a_i$ and is therefore  equal to $\ZZ\cap D(\b)=supp_\ZZ(\b)$.
\end{proof}

\section{ Simplicial complex of generalized cluster tilting sets}

This collaboration began with the surprising discovery of combinatorial connections between cluster tilting objects and cluster categories~\cite{[BMRRT]}, the supports of semi-invariants~\cite{[DW1]}, and chain resolutions of nilpotent groups~\cite{[IO]}.  In this section we begin to unveil these connections, first reminding the reader of some basic results and definitions concerning cluster categories in section~\ref{ClCat}.  In section~\ref{ClCatIntVec} we relate objects in cluster categories to Schur roots and negative projective roots.  With these connections established, in section~\ref{ContMap} we construct a continuous monomorphism from a subcomplex of virtual semi-tilting sets onto a dense subset of the $(n-1)$-sphere.

In the next section, at last, we restrict our considerations from general quivers to Dynkin quivers, and prove that in this special case the above continuous monomorphism is a homeomorphism, providing a simplicial decomposition of the sphere with codimension one skeleta given by the domains of the virtual semi-invariants.



\subsection{ Cluster categories and cluster tilting objects} \label{ClCat}

Let $Q$ be a finite quiver with no oriented cycles. The associated  cluster category $\mathcal C_Q$ was defined in \cite{[BMRRT]} as a special orbit category of the  associated bounded derived category $\mathcal D^b_Q$ in the following way.

Let $\mmod\k Q$ be the category of finitely generated modules over the path algebra $\k Q$ and let $\tau$ be the Auslander-Retien translation functor. Since $\k Q$ is hereditary, $\tau$ is a functor  
$\mmod \k Q \to \mmod \k Q$ which induces an equivalence of full subcategories
$$\{ \k Q\text{-modules{ \small w\!/\!o }projective summands}\} \xrarrow{\tau} \{ \k Q \text{-modules{ \small w\!/\!o }injective summands}\}.$$


An important fact is that the Auslander-Reiten functor can be extended to an auto-equivalence of the associated derived category $\mathcal D^b_Q$ which we will describe now.

Let $\mathcal D^b_Q:=\mathcal D^b(\mmod \k Q)$ be the derived category of bounded complexes in $\mmod \k Q$.
 Instead of recalling the general definition of the derived categories, we will describe objects and morphisms, which is quite easy since the algebra $\k Q$ is hereditary: the indecomposable complexes are isomorphic to stalk complexes, hence all indecomposable objects can be described as shifts of the indecomposable modules: 
$$\ind \mathcal D^b_Q=
\cup_{i\in \ZZ}(\ind \k Q)[i].$$
The morphisms in  $\mathcal D^b_Q$ can also be easily described:
for all $M,N \in \mmod \k Q$\\
$\Hom_{\mathcal D^b_Q}(M,N) = \Hom_Q(M,N);\ \ \ \  \Hom_{\mathcal D^b_Q}(M,N[1]) = \Ext^1_Q(M,N);$\\
$\Hom_{\mathcal D^b_Q}(M,N[i])\!=\!0, \ i\!\neq\!0,\!1; \ \ \Hom_{\mathcal D^b_Q}(X,Y)\! =\!\Hom_{\mathcal D^b_Q}(X[i],Y[i]), \ i\!\in\!\ZZ, X,\!Y\!\!\in \!\mathcal D^b_Q.$

Let $\mathcal D^b_Q \xrarrow{\tau} \mathcal D^b_Q$ be the automorphism of the category induced by the Auslander-Reiten translation functor, which we also  call the Auslander-Reiten, or AR, functor. 
Then  the composition functor $\mathcal D^b_Q \xrarrow{[1]} \mathcal D^b_Q \xrarrow{\tau^{-1}} \mathcal D^b_Q$ is an auto-equivalence of $D^b_Q$. 


\begin{defn}\cite{[BMRRT]} \emph{The cluster category} $\mathcal C_Q$ for a quiver $Q$ is the orbit category 
$$\mathcal C_Q:= \mathcal D^b_Q/(\tau^{-1}[1]) $$ 
of the derived category $\mathcal D^b_Q$,  under the action of $\tau^{-1}[1]$.
\end{defn}
\begin{rem} A set of representatives of the indecomposable $\mathcal C_Q$ objects, which are $(\tau^{-1}[1])$-orbits, may be chosen to be in
$$\ind \k Q \cup \{P(v)[1]\}_{v\in Q_0},$$ 
the set of indecomposable $\k Q$-modules and shifts $P(v)[1]$ of the indecomposable projective $\k Q$-modules $P(v)$.
\end{rem}

Some particularly important objects in the cluster category are the \emph{cluster tilting objects}, which are essential in the ``cluster algebra/cluster category'' relations. Their definition extends the classical definition \cite{[HU]}, \cite{[U]} of a \emph{tilting module} as a module $T= \smallcoprod _{i=1}^nT_i$
 of nonisomorphic indecomposable modules $T_i$ so that $\Ext(\smallcoprod T_i, \smallcoprod T_i)=0$.
\begin{defn} An object $T= \smallcoprod _{i=1}^nT_i$ of the cluster category $\mathcal C_Q$ is called a \emph{cluster tilting object} if $\Ext^1_{\mathcal C_Q}(T,T)=0$, where $T_i$ are indecomposable and pairwise non-isomorphic.
\end{defn}

\subsection{ A relation between objects of cluster categories and integral vectors} \label{ClCatIntVec}
To each object of $\mathcal C_Q$ which has a representative in the module category $\mmod \k Q$ we associate the dimension vector $\ul\dim \,M\in \NN^n\times 0\subset \ZZ^n \times\ZZ$, and to  each object of $\mathcal C_Q$ which has a representative shifted projective we associate the vector $(\ul\dim \,P)[1]\in \NN^n\times 1\subset \ZZ^n \times\ZZ$.

\begin{defn}  A \emph{Schur representation} is a $\k Q$-module $M$ so that $\End_{Q}(M)\!=\!\k$.
\end{defn}

To translate from cluster category objects to integral vectors, first we make the translation from modules to Schur roots, using the following theorem.

\begin{thm} \cite{[Kac82]} A vector $\a\in\NN^n$ is a Schur root if and only if there exists a Schur representation $M$ with $\ul\dim\,M=\a$.

\begin{rem} (1) The Auslander-Reiten functor $\tau$,
on the level of dimension vectors, $\t:\ZZ^n\cong\ZZ^n$ is given by $\t=-E^{-1}E^t$. \\
(2) 
$\<\a,\b\>=\b^tE^t\a=-\<\b,\t\a\>$
 \\
(3) $\<\t\a,\t\b\>=\<\a,\b\>$\\
(4) By the properties of the translation functor, this linear map  sends Schur roots and negative projective roots to Schur roots and negative injective roots.
\end{rem}

\end{thm}
\subsection{ Virtual semi-tilting sets} 
In this section we consider only the dimension vectors of certain indecomposable modules and shifted indecomposable projective modules, which is a subset of the
representatives of the indecomposable objects of cluster category $\mathcal C_Q$. Specifically, we consider Schur roots and shifted projective roots $p(v)[1]=(E^t)^{-1}(e_v)[1]$.

We prove the necessary corollaries to the Generic Decomposition Theorem in order to construct the Tilting Triangulation of Section~\ref{finitecase} and exhibit its properties.


\begin{defn}\label{VirSTilt}
A {\em partial virtual semi-tilting set} for a quiver $Q$ with $n$ vertices is a collection of distinct Schur roots and shifted indecomposable projective roots
$$\{\b_{1},\!\cdots\!,\b_{k}, p(v_{k+1})[1],\!\cdots\!,p(v_m)[1]\}$$ 
 with $ext_{Q} (\b_i,\b_j)=0$ and $hom_Q(p(v_i),\b_j)=(\b_j)_{v_i}=0$ for all $i\neq j$.\\
 A {\em virtual semi-tilting set} is partial virtual semi-tilting set with $n$ elements.
\end{defn}
\begin{rem}  We point out similarities and differences between cluster tilting objects and virtual semi-tilting sets.
\begin{itemize}
\item Both cluster tilting objects and virtual semi-tilting sets may include shifted projectives (which is not the case with the classical tilting modules).
\item Both cluster tilting objects and virtual semi-tilting sets must have all of their indecomposable components  not extend each other.
\item A module $M$ and a shifted projective $P[1]$ do not extend each other if and only if $\Hom_Q(P,M)=0$. This agrees with $hom_Q(p(v_i),\b_j)=0$, the same condition in the definition of partial virtual semi-tilting sets \ref{VirSTilt}.
\item The prefix \emph{semi-} is used to emphasize that $ext_Q(\b_i,\b_i)$ may be nonzero, while for any tilting or cluster tilting object we have $\Ext_{\mathcal C_Q}(T_i,T_i)=0$.
\end{itemize}
\end{rem}
\begin{eg}
Some examples of virtual semi-tilting sets are:\\
 (1) The set of indecomposable projective roots\\
 (2) The shifted projective roots \\
 (3) The roots corresponding to the indecomposable injective modules\\
 (4) Each null root for  an extended Dynkin diagram forms a partial virtual semi-tilting set, but no module of that dimension can be a partial cluster tilting object.\\
 (5) In the example \ref{Example} the following sets of roots form virtual semi-tilting sets, actually maximal such sets:
 $\{(-1,-1,-2),(0,1,1)\}$ and $\{\a=(1,2,2)\}$. Note that $ext(\a,\a)\neq0.$
  \end{eg}
 \begin{rem}
 (a) Note that every subset of a virtual semi-tilting set is a partial virtual semi-tilting set. Thus:\\ 
(b) The partial virtual semi-tilting sets form a simplicial complex. 

Recall that a \emph{simplicial complex} is a collection $K$ of finite nonempty sets $\delta$ called \emph{simplicies} so that any nonempty subset of a simplex is also a simplex. A simplex of $K$ with $p+1$ elements is called a $p$-\emph{simplex} of $K$ and the set of $p$-simplices of $K$ is denoted $K_p$.
\end{rem}

\subsection{ Complex of virtual semi-tilting sets} \label{ContMap} We will construct a simplicial complex $\cT(Q)$ and a subcomplex $\cT'(Q)$. We will see that $\cT(Q)$ is $(n-1)$-dimensional, where $n$ is the number of vertices of $Q$, and there is a continuous mapping of the geometric realization of $\cT(Q)$ to the standard $(n-1)$-sphere $S^{n-1}$. When restricting to the subcomplex $\cT'(Q)$ we will get a continuous monomorphism
\[
	\ll:|\cT'(Q)|\to S^{n-1}
\]
whose image is dense. This implies that, if $\cT'(Q)$ is a finite simplicial complex, $\ll$ is a homeomorphism.

\begin{defn} Let $Q$ be a quiver w/o oriented cycles.  The {\em complex of virtual semi-tilting sets}, $\cT(Q)$, is the simplicial complex whose simplices are the partial virtual semi-tilting sets of the Schur roots and the shifted indecomposable roots.
\end{defn}

We will use the Generic Decomposition Theorem~\ref{thm:generic decomposition} and the following result of Schofield to show that this simplicial complex is $(n-1)$ dimensional.

\begin{thm}\cite{[S1]} \label{decomposition of ma}
Let $Q$ be a quiver with $n$ vertices and no oriented cycles. \\
(1) Any multiple $m\a$ of a Schur root $\a$ is either a Schur root or decomposes generically as a sum of $m$ copies of $\a$.\\
(2) If $\a=\sum \b_i$ is a generic decomposition of $\a\in\NN^n$ then $m\a=\sum (m\b_i)$ is the generic decomposition of $m\a$ where $(m\b_i)$ denotes either a single Schur root or a sum of $m$ copies of $\b_i$ in the case when $m\b_i$ is not a Schur root.

\end{thm}

\begin{prop}
If $\{\b_{1},\!\cdots\!,\b_{k}, p(v_{k+1})[1],\!\cdots\!,p(v_m)[1]\}$ is a partial virtual semi-tilting set then the corresponding subset 
$\{\b_{1},\!\cdots\!,\b_{k}, -p(v_{k+1}),\!\cdots\!,-p(v_m)\} \subset \ZZ^n$ is linearly independent over $\QQ$. In particular, it has at most $n$ elements.
\end{prop}

\begin{proof}
Any rational linear relation on the vectors $p(v_i),\b_j$ gives an integral linear relation by multiplying by the common denominators of the rational coefficients. Collecting terms with positive and negative coefficients we get an equation of the form
\[
	\a=\sum (n_i \b_i)-(E^t)^{-1}\g=\sum (m_j \b_j)-(E^t)^{-1}\g'
\]
where $n_i,m_j\geq0$ and $\g,\g'$ has support disjoint from any of the $\b_i,\b_j$. This gives two different generic decompositions of the same dimension vector contradicting Theorem~\ref{thm:generic decomposition}. To see that these are generic decompositions of $\a$ we use theorem \ref{decomposition of ma} and the observation that $ext_Q(m_i\b_i,m_j\b_j)\le m_im_jext_Q(\b_i,\b_j)=0$.
\end{proof}

\begin{cor} If $Q$ has $n$ vertices then
the simplicial complex $\cT(Q)$ is $(n-1)$ dimensional.
\end{cor}

Let $K$ be a simplicial complex.  Let $K_0$ be the vertex set, that is, the set of $0$-simplices. The \emph{geometric realization} $|K|$ of $K$ is defined to be the subspace of the infinite dimensional vector space $\RR^{K_0}$ consisting of all vectors $x=\sum t_i v_i$, $t_i\in [0,1]$, $v_i\in K_0$ with the property that $\sum t_i=1$ and set of all vertices $v_i$ with nonzero coefficient is a simplex $\delta\in K$. For a fixed $\delta\in K$ the set of all such $x$, i.e., all $x\in|K|$ which are positive linear combinations of the vertices of $\delta$, is called the \emph{open simplex} $e^\delta$ of $\delta$. Note that $|K|$ is by definition a disjoint union of open simplices. The closure of an open simplex is the \emph{closed simplex} $\Delta^\delta$ of $\delta$ which is the set of all $x\in|K|$ which are nonnegative linear combinations of the vertices of $\delta$.

If $K_0$ is finite we take the usual Euclidean topology on $\RR^{K_0}$. When it is infinite we take the \emph{weak topology} which is the direct limit of all $\RR^S$ where $S$ runs over all finite subsets of $K_0$. This is the weakest topology (having the fewest open sets) on $|K|$ with the property that a mapping $\ll:|K|\to X$ is continuous if and only if it is continuous on every closed simplex $\Delta^\delta$.

Let $\LL:\mathcal T_0(Q)\to \RR^n$ be the mapping which sends each Schur root $\b$ to itself and $p[1]$ to $-p$. Then $\LL$ extends to a continuous mapping $|\LL|:|\cT(Q)|\to \RR^n$ given by
\[
	|\LL|\left(\sum  t_j\b_j +\sum t_ip(v_i)[1]\right)= \sum t_j\b_j -\sum t_i p(v_i).
\]
The proposition above implies that $0$ is not in the image of this mapping. Therefore, we can normalize to get a continuous mapping  $\ll:|\cT(Q)|\to S^{n-1}$
\[
	\ll\left(\sum  t_j\b_j +\sum t_ip(v_i)[1]\right):=\frac{\sum t_j\b_j -\sum t_i p(v_i)}{
	\length{\sum t_j\b_j -\sum t_i p(v_i)}
	}
\]
We would like this mapping to be a monomorphism. However, if there are Schur roots $k\b,m\b$ which are multiples of the same $\b\in\NN^n$ then clearly $\ll(k\b)=\ll(m\b)$. So, $\ll$ may not be a monomorphism. To remedy this we restrict to a subcomplex $\cT'(Q)$ of $\cT(Q)$ which we now define.

We say that a Schur root $\b$ is {\em minimal} if its coefficients are relatively prime, i.e., $\b$ is not a positive integer multiple of another vector in $\NN^n$. By Schofield's theorem mentioned above, every Schur root is a multiple of a minimal Schur root since any decomposition of $\b$ would result in a decomposition of $m\b$. 

The following corollary is a consequence of the virtual generic decomposition theorem \ref{thm:generic decomposition}.

\begin{cor}\label{uniquedim}
Every nonzero vector $x\in\QQ^n$ can be written uniquely as a linear combination
\[
	x= \sum x_j \b_j-\sum x_i p(v_i)
\]
where $x_i,x_j>0$ are positive rational numbers, $\b_j$ are minimal Schur roots and $p(v_i)[1]$ are shifted indecomposable projective roots forming a partial virtual semi-tilting set.
\end{cor}

\begin{proof}
Multiply by a sufficiently large integer $k$ to get $kx\in \ZZ^n$. Then apply Theorem~\ref{thm:generic decomposition} to $kx$ to get a generic decomposition
\[
	kx=\sum (m_j\b_j)-\sum m_i p(v_i)
\]
where we have collected repeating factors $\b_j$ using Schofield's notation.
Next, write each summand $(m_j\b_j)$ as a multiple of a minimal Schur root. Then divide by $k$ to get the desired rational linear decomposition of $x$.

To prove uniqueness, suppose we have two rational decompositions of $x$. Then we get integer decompositions of say, $kx, mx$ giving two generic decompositions of $mkx$ (using  \ref{decomposition of ma}) which is a contradiction.
\end{proof}

\begin{thm}\label{dense mono} The restriction of $\ll:|\cT(Q)|\to S^{n-1}$ to the subcomplex $\cT'(Q)$ of $\cT(Q)$ consisting of all simplices whose vertices are either minimal Schur roots or shifted indecomposable projective roots gives a continuous mapping:
\[
	\ll':|\cT'(Q)|\to S^{n-1}
\]
which is a monomorphism whose image is dense in the standard Euclidean topology on $S^{n-1}$.
\end{thm}

\begin{proof}
The uniqueness statement of \ref{uniquedim} implies that $\ll'$ is a monomorphism. The existence part of \ref{uniquedim} implies that the image of $\ll'$ contains the image of $\QQ^n-\{0\}$ under the normalization map $\cdot/||\cdot||:\QQ^n-\{0\}\to S^{n-1}$ whose image is dense.
\end{proof}


\section{Semi-invariants and the cluster tilting triangulation associated to a Dynkin quiver}\label{finitecase}\label{sec8}

Until now, $Q$ was an arbitrary quiver without oriented cycles.  We now assume $Q$ is a Dynkin quiver, i.e. a simply laced Dynkin diagram with any orientation of its edges. We define the {\em cluster tilting triangulation} associated to a Dynkin diagram which triangulates the sphere via the complex of cluster tilting sets, and show that the supports of the semi-invariants of the quiver comprise the codimension-one skeleton.

\subsection{ Cluster tilting triangulation}

Since $Q$ is Dynkin, the set of Schur roots equals the set of positive roots $\Phi_+$ of the Euler form. Also we recall that if $\b$ is a positive root then there is a unique indecomposable module of dimension $\b$ up to isomorphism and the set of all elements of $R(\b)$ isomorphic to this module is open. The Schur roots are all minimal. They and the shifted projectives form a finite set we denote by $\Phi_+'$.

\begin{defn}
The {\em cluster tilting complex} of the Dynkin diagram $Q$ is defined to be the simplicial complex $\cT(Q)=\cT'(Q)$ with vertex set $\Phi_+'$ so that the faces of $\cT(Q)$ are the virtual semi-tilting sets.
\end{defn}


\begin{thm}
For a Dynkin quiver, the geometric realization of the cluster tilting complex is homeomorphic to the $(n-1)$-sphere. Furthermore, a homeomorphism $\ll:|\cT(Q)|\to S^{n-1}$ is given by
\[
	\ll\left(\sum  t_j\b_j+\sum t_ip(v_i)[1] \right):=\frac{\sum t_j\b_j -\sum t_i p(v_i)}{
	\length{\sum t_j\b_j -\sum t_i p(v_i)}
	}
\]
\end{thm}

\begin{proof} Since all Schur roots are minimal, $\cT'(Q)=\cT(Q)$. Therefore, Corollary \ref{dense mono} applies to $\ll=\ll'$ to show that $\ll:|\cT(Q)|\to S^{n-1}$ is a continuous monomorphism with dense image. However, $\cT(Q)$ is a finite complex. So, $|\cT(Q)|$ is compact. This means that $\ll$ is a homeomorphism onto its image which must also be compact and therefore all of $S^{n-1}$.
\end{proof}

In any triangulation of a closed manifold, a codimension one simplex is a face of exactly two simplices of maximal dimension. This gives the following corollary, which is a special case of a theorem from \cite{[BMRRT]}.

\begin{cor}
In the Dynkin case, any almost complete generalized cluster tilting set of roots is contained in exactly two complete virtual tilting sets.
\end{cor}

In the finite case the supports of the semi-invariants are easy to describe.

\begin{lem}\label{lem:perp}
Let $\a,\b$ be positive roots of the Dynkin quiver $Q$. Then $\a\in D(\b)$ if and only if $\<\a,\b\>=0$.
\end{lem}

\begin{proof}
If $\a\in D(\b)$ then $\<\a,\b\>=0$ by definition. Conversely,
if $\<\a,\b\>=0$ then $hom_Q(\a,\b)=ext_Q(\a,\b)=0$ since they cannot both be nonzero. This implies that $hom_Q(\a,\b')=0$ for any $\b'\into\b$. So, $\<\a,\b'\>\leq0$ and $\a\in D(\b)$. 
\end{proof}

\begin{thm}\label{thm:perp}
Let $\b$ be a positive root of $Q$. Then $D(\b)\subset\RR^n$ is the set of all nonnegative real linear combinations of positive roots $\a$ so that 
$\langle\a,\b\rangle=0$ and negative projective roots $-p(v_i)$ so that $\langle p(v_i),\b\rangle =\b_{v_i}=0$.
\end{thm}


\begin{proof}
Since $D(\b)$ is given by homogeneous linear equations and inequalities with integer coefficients, it suffices to prove the theorem in $\ZZ^n$ instead of $\RR^n$. So let $\a\in\ZZ^n\cap D(\b)$. By the Virtual Saturation Theorem~\ref{virsatthm}  and its corollary this set is the same as the support of $\cc_V$ where $V$ is the unique indecomposable representation with dimension $\b$. Proposition \ref{homext} implies that there is a module $M$ and a projective module $P$ over $\k Q$, such that $\a ={\ul\dim}M-{\ul\dim}P$, and so that $\Hom_Q(M,V)=\Ext^1_Q(M,V)=0$ and $\Hom_Q(P,V)=0$. But then the same holds for all indecomposable direct summands $M_i$ of $M$ and all indecomposable direct summands $-P_j$ of $-P$. 
Thus the corresponding roots $\a_i$ lie in $D(\b)$. So $\a=\sum\ul\dim M_i - \sum\ul\dim P_j$ is a positive linear combination of the required roots.
\end{proof}


\begin{cor}\label{cor:perp}
For any $x\in D(\b)$ there is a virtual semi-tilting set $\{\a_j\}$ all of whose elements lie in $D(\b)$ so that $x$ is a nonegative linear combination of the $\a_j$.
\end{cor}

\begin{proof}
Since the set of all $x$ satisfying this condition is closed, it suffices to show that it holds for a dense subset of $D(\b)$. So, we may assume that $x\in\QQ^n$. Multiplying by the denominator we may assume $x\in\ZZ^n$. Now, repeat the last step of the proof of Theorem~\ref{thm:perp}.
\end{proof}

Our Main Theorem identifies the codimension one skeleton of $\cT(Q)$ with the supports of semi-invariants.

\begin{thm}\label{TTDSI}
The image in $S^{n-1}$ of the $n-2$ skeleton of $\cT(Q)$ under the homeomorphism $\ll$ is the union of supports of semi-invariants:
\[
	\ll\left(|\cT(Q)^{n-2}|\right)=\bigcup_{\b\in\Phi_+}D(\b)\cap S^{n-1}
\]
\end{thm}

\begin{proof}
The statement is equivalent to saying that $\cup D(\b)$ is equal to the union of rays eminating from $0$ and passing through the $n-2$ skeleton of $\cT(Q)$. Corollary \ref{cor:perp} implies that each $D(\b)$ is contained in this union of rays. To prove the converse it suffices to show that any partial tilting set $\{\a_1,\cdots,\a_{n-1}\}$ is contained in $D(\b)$ for some positive root $\b$. By Lemma \ref{lem:perp} this is equivalent to saying that $\langle \a_i,\b\rangle=0$ for each $\a_i$. This is trivially true when $n=1$. So we may assume that $n\geq2$ and the statement holds for all Dynkin quivers with fewer vertices.

Suppose for a moment that one of the $\a_i$ is a shifted projective $p(v)[1]$. If $v$ is minimal then all other roots $\a_j$ have support disjoint from $v$. So, we can delete the vertex $v$ and delete the root $p(v)[1]$ to obtain a partial tilting set on $Q'$, the subquiver of $Q$ given by deleting the vertex $v$ and all arrows to and from $v$. By a counting argument we see that this consists of a partial tilting set on one of the components of $Q'$ and a complete tilting set on the other components. By induction on $n$ there is a positive root $\b$ on the first component so that $\<|\a_i|,\b\>=0$ all the $i$. This gives the desired root for $Q$ proving the theorem in this case.

If none of the $\a_i$ is a shifted projective we use the inverse translation $\t^{-1}$. Let $m>0$ be minimal so that at least one $\t^{-m}\a_i$ is a shifted projective. By the previous case there is a positive root $\b$ so that $\<|\t^{-m}\a_i|,\b\>=0$ for all $i$. Then $\<\a_i,|\t^m\b|\>=0$ for all $i$. So $\{\a_i\}\subset D(|\t^m\b|)$ as claimed.
\end{proof}


\end{document}